\documentclass{article}

\usepackage{microtype}
\usepackage{graphicx}
\usepackage{booktabs} 

\usepackage{graphicx}   
\usepackage{mathtools}
\usepackage{subfigure}

\usepackage[colorinlistoftodos,prependcaption,textsize=tiny]{todonotes}

\usepackage{amsthm}
\usepackage{amsmath, amssymb}

\newcommand{\argmin}{\mathop{\rm argmin}}

\newcommand{\less}{\leqslant}
\newcommand{\gre}{\geqslant}

\newcommand{\defn}{\ensuremath{: \, =}}
\newcommand{\xstar}{\ensuremath{x^*}}
\newcommand{\real}{\ensuremath{\mathbb{R}}}
\newcommand{\Exs}{\ensuremath{\mathbb{E}}}

\newtheorem{theorem}{Theorem}
\newtheorem{lemma}{Lemma}[section]
\newtheorem{lemmaappendix}{Lemma}

\newtheorem{remark}{Remark}

\DeclareMathOperator{\trace}{trace}

\usepackage{hyperref}

\usepackage[accepted]{icml2020}

\newcommand{\citep}{\cite}
\newcommand{\citet}{\cite}

\icmltitlerunning{Limiting Density of SRHT Matrices and Applications}

\begin{document}

\twocolumn[
\icmltitle{Optimal Randomized First-Order Methods for Least-Squares Problems}

\icmlsetsymbol{equal}{*}

\begin{icmlauthorlist}
\icmlauthor{Jonathan Lacotte}{}
\icmlauthor{Mert Pilanci}{}
\end{icmlauthorlist}



\icmlkeywords{Random matrices; Subsampled randomized Hadamard transform; Limiting spectral distribution; Least-squares optimization.}

\vskip 0.3in
]


\begin{abstract}
We provide an exact analysis of a class of randomized algorithms for solving overdetermined least-squares problems. We consider first-order methods, where the gradients are pre-conditioned by an approximation of the Hessian, based on a subspace embedding of the data matrix. This class of algorithms encompasses several randomized methods among the fastest solvers for least-squares problems. We focus on two classical embeddings, namely, Gaussian projections and subsampled randomized Hadamard transforms (SRHT). Our key technical innovation is the derivation of the limiting spectral density of SRHT embeddings. Leveraging this novel result, we derive the family of normalized orthogonal polynomials of the SRHT density and we find the optimal pre-conditioned first-order method along with its rate of convergence. Our analysis of Gaussian embeddings proceeds similarly, and leverages classical random matrix theory results. In particular, we show that for a given sketch size, SRHT embeddings exhibits a faster rate of convergence than Gaussian embeddings. Then, we propose a new algorithm by optimizing the computational complexity over the choice of the sketching dimension. To our knowledge, our resulting algorithm yields the best known complexity for solving least-squares problems with no condition number dependence.
\end{abstract}

\section{Introduction}
\label{SectionIntroduction}

We study the performance of a randomized method, namely, the Hessian sketch~\citep{pilanci2016iterative}, in the context of (overdetermined) least-squares problems,
\begin{align}
\label{EqnMain}
    x^* \defn \argmin_{x \in \real^d} \left\{f(x) \defn \frac{1}{2} \|Ax-b\|^2\right\}\,,
\end{align}
where $A \in \real^{n \times d}$ is a given data matrix with $n \gre d$ and $b \in \real^n$ is a vector of observations. For simplicity of notations, we will assume throughout this work that $\text{rank}(A)=d$. 

Many works have developed randomized algorithms~\citep{avron2010blendenpik, rokhlin2008fast, drineas2011faster, pilanci2015randomized} for solving~\eqref{EqnMain}, based on sketching methods. The latter involve using a random matrix $S \in \real^{m \times n}$ to project the data $A$ and/or $b$ to a lower dimensional space ($m \ll n$), and then approximately solving the least-squares problem using the sketch $SA$ and/or $Sb$. The most classical sketch is a matrix $S \in \real^{m \times n}$ with independent and identically distributed (i.i.d.) Gaussian entries $\mathcal{N}(0,m^{-1})$, for which forming $SA$ requires in general $\mathcal{O}(mnd)$ basic operations (using classical matrix multiplication). This is larger than the cost $\mathcal{O}(nd^2)$ of solving~\eqref{EqnMain} through standard matrix factorization methods, provided that $m \gre d$. Another well-studied embedding is the (truncated) $m \times n$ Haar matrix $S$, whose rows are orthonormal and with range uniformly distributed among the subspaces of $\real^n$ with dimension $m$. However, it requires time $\mathcal{O}(n m^2)$ to be formed, through a Gram-Schmidt procedure, which is also larger than $\mathcal{O}(nd^2)$. An alternative embedding which verifies orthogonality properties is the SRHT~\citep{ailon2006approximate}, which is based on the Walsh-Hadamard transform. Due to the recursive structure of the latter, the sketch $SA$ can be formed in $\mathcal{O}(nd \log m)$ time, so that the SRHT is often viewed as a standard reference point for comparing sketching algorithms. 

It has been observed in several contexts that random projections with i.i.d.~entries degrade the performance of the approximate solution compared to orthogonal projections~\citep{mahoney2011randomized, mahoney2016structural, drineas2016randnla, dobriban2019asymptotics}. Consequently, along with computational considerations, this suggests to consider the SRHT over Gaussian or Haar projections. On the other hand, in order to pick optimal algorithm's parameters, it is usually necessary to have a tight characterization of the spectrum of $C_S \defn U^\top S^\top S U$, where $U$ is the matrix of left singular vectors of $A$, which is the case for Gaussian embeddings.

Using the standard prediction (semi-)norm $\|A(\widetilde x - x^*)\|^2$ as the evaluation criterion for an approximate solution $\widetilde x$, iterative methods (e.g., gradient descent or the conjugate gradient algorithm) have time complexity which usually scales proportionally to the condition number $\kappa$ of the matrix $A$ -- defined as the ratio between the largest and smallest singular values of $A$ --, and this becomes prohibitively large when $\kappa \gg 1$. To address the latter issue, we introduce a pre-conditioning method, namely, the Hessian sketch~\cite{pilanci2016iterative}, which approximates the Hessian $H=A^\top A$ of $f(x)$ by $H_S = A^\top S^\top S A$. Then, we consider the following class of pre-conditioned first-order methods (also referred to as a quasi-Newton method in the optimization literature),
\begin{align}
\label{EqnFirstOrder}
    x_t \in x_0 + H_S^{-1} \cdot \text{span}\left\{ \nabla f(x_0), \dots, \nabla f(x_{t-1})\right\}\,,
\end{align}
Several versions of~\eqref{EqnFirstOrder} have been studied. For instance, it has been recently shown by~\citep{ozaslan2019iterative} that the Heavy-ball update
\begin{align}
\label{EqnIHSUpdate}
    x_{t+1} = x_{t} - \mu_{t} H_{S}^{-1} \nabla f(x_t) + \beta_{t} (x_t\!-\!x_{t-1})
\end{align}
yields a sequence of iterates whose convergence rate does not depend on the spectrum of $A$, but only on the concentration of the matrix $C_S$ around the identity, and they show that this convergence rate is equal to the ratio $d/m$ both for Gaussian and SRHT embeddings. Notably, this rate does not depend on the sample size $n$. For a Gaussian embedding, this makes intuitively sense since the limiting spectral distribution of $C_S$ is the Marchenko-Pastur law~\citep{marchenko1967distribution} with scale parameter $\rho$, edge eigenvalues $a \!=\! (1\!-\!\sqrt{\rho})^2$ and $b \!=\! (1\!+\!\sqrt{\rho})^2$, and density 
\begin{align}
\label{EqnMarchenkoPasturDensity}
	\mu_{\rho}(x) = \frac{\sqrt{(b-x)_+(x-a)_+}}{2\pi \rho x}\,,
\end{align} 
where $y_+=\max\{y,0\}$, and it does not depend on the sample size $n$ but only on the limit ratio $\rho \defn \lim d/m$. However, for a SRHT embedding, it is unclear if the dimension $n$ affects the best achievable convergence rate. In a related vein, the authors of~\citet{lacotteiterative20} considered the Heavy-ball update~\eqref{EqnIHSUpdate} where at each iteration the sketching SRHT matrix $S\!=\!S_t$ is refreshed (i.e., re-sampled independently of $S_0,\dots,S_{t-1}$) so that $H_S\!=\!H_{S_t}$ is also re-computed. They show that Haar and SRHT embeddings yield the same convergence rate $\rho_h^\mathrm{ref}\!\defn\!\rho \cdot \frac{\xi (1-\xi)}{\gamma^2+\xi-2\gamma\xi}$, which indeed depends on the three relevant dimensions $m,d,n$ through the aspect ratios $\rho$, $\gamma\defn\lim d/n$ and $\xi\defn\lim m/n$, which is always strictly smaller than $\rho$, the convergence rate one would obtain with refreshed Gaussian embeddings~\cite{lacotte2019faster}.

In this work, we consider a definition of the SRHT slightly different than its classical version~\cite{ailon2006approximate}, which has been introduced in~\cite{dobriban2019asymptotics, liu2019ridge}. For an integer $n=2^p$ with $p \gre 1$, the Walsh-Hadamard transform is defined recursively as $H_n = \frac{1}{\sqrt{2}} \begin{bmatrix} H_{n/2} & H_{n/2} \\ H_{n/2} & -H_{n/2} \end{bmatrix}$ with $H_1 = 1$. Our transform $A \mapsto S A$ first \emph{randomly permutes} the rows of $A$, before applying the classical transform. This has negligible cost $\mathcal{O}(n)$ compared to the cost $\mathcal{O}(nd \log m)$ of the matrix multiplication $A \mapsto SA$, and \emph{breaks the non-uniformity} in the data. That is, we define the $n\times n$ subsampled randomized Hadamard matrix as $S = B H_n D P$, where $B$ is an $n \times n$ diagonal sampling matrix of i.i.d.~Bernoulli random variables with success probability $m/n$, $H_n$ is the $n \times n$ Walsh-Hadamard matrix, $D$ is an $n \times n$ diagonal matrix of i.i.d.~sign random variables, equal to $\pm 1$ with equal probability, and $P\in\real^{n\times n}$ is a uniformly distributed permutation matrix. At the last step, we discard the zero rows of $S$, so that it becomes an $\widetilde m \times n$ orthogonal matrix with $\widetilde m \sim \mathrm{Binomial}(m/n,n)$, and the ratio $\widetilde m / n$ concentrates fast around $\xi$ while $n \to \infty$. Although the dimension $\widetilde m$ is random, we refer to $S$ as an $m \times n$ SRHT matrix.

We will focus exclusively on (pre-conditioned) \textit{first-order methods} of the form~\eqref{EqnFirstOrder} with a fixed embedding $S$, and our goal is to answer the following questions. What are the best achievable convergence rates for, respectively, Gaussian and SRHT embeddings? What are the corresponding optimal algorithms? How do these rates compare to each other and to that of state-of-the-art randomized iterative methods for solving~\eqref{EqnMain}?

\subsection{Technical background, notations and assumptions}
\label{SectionMathematicalBackground}

We will assume that $\lim_{n \to \infty} \frac{d}{n} = \gamma \in (0,1)$, $\lim_{n \to \infty} \frac{m}{n} = \xi \in (\gamma, 1)$ and $\rho = \frac{\gamma}{\xi} \in (0,1)$. We denote $\|z\| \equiv \|z\|_2$ the Euclidean norm of a vector $z$, $\|M\|_2$ the operator norm of a matrix $M$, and $\|M\|_F$ its Frobenius norm. Given a sequence of iterates $\{x_t\}$, we denote the error at time $t$ by $\Delta_t = U^\top A (x_t - \xstar)$. Note that $\|\Delta_t\|^2 = \|A(x_t-\xstar)\|^2$. Our evaluation criterion is the error $\lim_{n \to \infty} \Exs[\|\Delta_t\|^2] / \Exs[\|\Delta_0\|^2]$, and we call its (asymptotic) rate of convergence the quantity $\limsup_{t \to \infty} \left(\lim_{n \to \infty} \Exs[\|\Delta_t\|^2] / \Exs[\|\Delta_0\|^2]\right)^{1/t}$. 

As we focus on infinite-dimensional regimes, our technical analysis is based on asymptotic random matrix theory, and we refer the reader to~\citep{bai2009spectral, paul2014random,yao2015large} for an extensive introduction to this field. For a random Hermitian matrix $M_n$ of size $n \times n$, the empirical spectral distribution (e.s.d.) of $M_n$ is the (cumulative) distribution function of its eigenvalues $\lambda_{1}, \hdots, \lambda_{n}$, i.e., $F_{M_n}(x) \defn \frac{1}{n} \sum_{j=1}^n \mathbf{1}\left\{\lambda_{j} \less x\right\}$ for $x \in \real$, which has density $f_{M_n}(x) = \frac{1}{n} \sum_{j=1}^n \delta_{\lambda_j}(x)$ with $\delta_{\lambda}$ the Dirac measure at $\lambda$. Due to the randomness of the eigenvalues, $F_{M_n}$ is random. The relevant aspect of some classes of large $n \times n$ symmetric random matrices $M_n$ is that, almost surely, the e.s.d.~$F_{M_n}$ converges weakly towards a non-random distribution $F$, as $n \to \infty$. This function $F$, if it exists, will be called the \emph{limiting spectral distribution} (l.s.d.)~of $M_n$. Key to our analysis is the notion of orthogonal polynomials, which are fundamental both in optimization~\cite{rutishauser1959theory} and in random matrix theory. We write $\real_t[X]$ the set of real polynomials with degree less than $t$, and $\real_t^0[X]$ the set of polynomials $P \in \real_t[X]$ such that $P(0)=1$. For a complex number $z \in \mathbb{C}$, we denote respectively by $\mathrm{Re}(z)$ and $\mathrm{Im}(z)$ its real and imaginary parts, and we use $\mathbb{C}_+$ for the complex numbers with positive imaginary parts, and $\mathbb{R}_+$ for the positive real numbers. For two sequences of real positive numbers $\{a_t\}$ and $\{b_t\}$, we write $a_t \asymp b_t$ if $\liminf \frac{a_t}{b_t} > 0$ and $\limsup \frac{a_t}{b_t} < \infty$.

We will assume that the first iterate $x_0$ is random such that $\Exs[x_0] = 0$, and, that the condition number of the matrix $U^\top A\Exs[x_0 x_0^\top] A^\top U + U^\top bb^\top U$ remains bounded as the dimensions grow. Essentially, this states that the condition number of $A$ does not degenerate to $+\infty$ as the dimensions grow.

\subsection{Overview of our results and contributions}

We have the following contributions.
\begin{enumerate}
    \item For Gaussian embeddings, we characterize the algorithm (Algorithm~\ref{AlgorithmOptimalFirstOrderGaussian}) which attains the infimum of the error $\lim_{n \to \infty} \Exs[\|\Delta^2_t\|] / \Exs[\|\Delta_0\|^2]$, and we show that it corresponds to the Heavy-ball method with constant step size $\mu_t \!=\! (1-\rho)^2$ and momentum parameter $\beta_t\!=\!\rho$. Further, we show that the infimum of the error is equal to $\rho^t$.
    \item For SRHT embeddings, we perform a similar analysis, and find the optimal first-order method (Algorithm~\ref{AlgorithmOptimalFirstOrderHaar}). Notably, it is a Heavy-ball update with non-constant step sizes and momentum parameters. Further, we show that its rate of convergence is $\rho_h\!\defn\! \rho \cdot \frac{1-\xi}{1-\gamma}$, which is always strictly smaller than $\rho$ and $\rho_h^\mathrm{ref}$, i.e., Algorithm~\ref{AlgorithmOptimalFirstOrderHaar} has uniformly better convergence rate than that of Gaussian embeddings or the Heavy-ball method with refreshed SRHT embeddings. Even though our theoretical results hold asymptotically, we verify empirically that our theoretical predictions hold, even for sample sizes $n \gtrsim 1000$, and that Algorithm~\ref{AlgorithmOptimalFirstOrderHaar} is faster in practice than the other aforementioned algorithms.
    \item We characterize explicitly the density $f_{h,r}$ of the l.s.d~of the matrix $\frac{n}{m} C_S$, which is given by
    \begin{align}
        f_{h,r}(x) = \frac{\sqrt{(\Lambda_{h,r}-x)_+(x-\lambda_{h,r})_+}}{2\pi \rho x (1-\xi x)}\,,
    \end{align}
    where the edge (i.e., extreme) eigenvalues are $\lambda_{h,r}=(\sqrt{1-\gamma}-\sqrt{(1-\xi)\rho})^2$ and $\Lambda_{h,r}=(\sqrt{1-\gamma}+\sqrt{(1-\xi)\rho})^2$. This characterization of the limiting density is of independent interest, as it might have several implications beyond least-squares optimization.
    \item  Finally, we show that Algorithm~\ref{AlgorithmOptimalFirstOrderHaar} has the best known complexity to solve~\eqref{EqnMain} with no condition number dependence.
\end{enumerate}
Except for the time complexity results, all our results regarding the SRHT hold exactly the same with Haar embeddings, since they both yield the same limiting spectral distributions.

\subsection{Other related work}

The design of optimal first-order methods for quadratic optimization problems has been recently considered in~\citet{pedregosaorthogonal19}. In contrast, they assume the data matrix to be random and they require its spectrum to be known beforehand, which is often impractical. On the other hand, our class of first-order methods applies a randomized pre-conditioning, so that only the spectral distribution of the matrix $C_S$ is required, and this is universal, i.e., independent of the spectrum of $A$. Therefore, by characterizing the l.s.d.~of $C_S$ for some classical embeddings, we are able to optimize the \emph{exact} error for any data matrix $A$. We note that existing methods do not directly minimize the error, but a worst-case upper bound as in Chebyshev iteration and Conjugate Gradient methods~\cite{saad2003iterative}.

Besides the Hessian sketch, there are many other efficient pre-conditioned iterative methods which aim to address the aforementioned conditioning issue, based on an SRHT sketch of the data (or closely related sketches based on the Fourier transform). Randomized \emph{right} pre-conditioning methods~\citep{avron2010blendenpik, rokhlin2008fast} compute first a matrix $P$ -- which itself depends on $SA$ -- such that the condition number of $AP^{-1}$ is $\mathcal{O}(1)$, and then apply any standard iterative algorithm to the pre-conditioned least-squares objective $\|AP^{-1}y-b\|^2$. SRHT sketches are also used for a wide range of applications across numerical linear algebra, statistics and convex optimization, such as low-rank matrix factorization~\citep{halko2011finding, witten2015randomized}, kernel regression~\citep{yang2017randomized}, random subspace optimization~\citep{lacotte2019high}, or, sketch and solve linear regression~\citep{dobriban2019asymptotics}. Hence, a refined analysis of the SRHT may also lead to better algorithms in these fields.

\section{Optimal first-order method for classical embeddings}

Let $S$ be an $m \times n$ Gaussian or SRHT embedding. Denote by $\mu$ the l.s.d.~of $C_S$. We say that a family of polynomials $\{R_k\}$ is orthogonal with respect to $\mu$ if $\int R_k R_\ell \,\mathrm{d} \mu \!=\! 0$ for any $k \!\neq\! \ell$. The next result establishes the link between polynomials and the pre-conditioned first-order methods~\eqref{EqnFirstOrder} we consider, and its proof is deferred to Appendix~\ref{ProofLemmaPolynomialRecursion}.
\begin{lemma}
\label{LemmaPolynomialRecursion}
Let $\{x_t\}$ be generated by some first-order method~\eqref{EqnFirstOrder}. Then, for any iteration $t \gre 0$, there exists a polynomial $p_t \in \real_t^0[X]$ such that $\Delta_t = p_t\!\left(C_S^{-1}\right) \cdot \Delta_0$. Further, it holds that 
\begin{align}
\label{EqnSquaredNorm}
    \lim_{n \to \infty} \frac{\Exs[\|\Delta_t\|^2]}{\Exs[\|\Delta_0\|^2]}  =  \int_\real p_t^2\!\left(\lambda^{-1}\right) \mathrm{d} \mu(\lambda)\,.
\end{align}
\end{lemma}
Thus, the best achievable error is lower bounded by the infimum of the following variational problem,
\begin{align}
\label{EqnPolynomialOptimizationProblem}
    \mathcal{L}^*_{\mu,t} \defn \min_{p \in \real^0_t[X]} F_\mu(p) \,,
\end{align}
where $F_\mu(p)\! \defn \!\int p^2\!\left(\lambda^{-1}\right) \, \mathrm{d}\mu(\lambda)$. Using the change of variable $x\!=\!1/\lambda$ and setting $\mathrm{d}\nu(x)=x^{-1}\mathrm{d}\mu\!\left(x^{-1}\right)$, we have that $F_\mu(p) \!=\! G_\nu(p)$ where $G_\nu(p) \!\defn\! \int p^2(x) \,\frac{1}{x} \mathrm{d}\nu(x)$. The optimal polynomial can be constructed by leveraging the following result.
\begin{lemma}
\label{LemmaOptimalSolutionPolynomialOptimizationProblem}
Let $\nu$ be some measure with bounded support in $(0,+\infty)$, and suppose that $\{\Pi_t\}$ is a family of orthogonal polynomials with respect to $\nu$ such that $\mathrm{deg}(\Pi_t)=t$ and $\Pi_t(0)=1$. Then, the polynomial $\Pi_t$ is the unique solution of the optimization problem $\min G_\nu(p)$ over $p \in \real_t^0[X]$.
\end{lemma}
\begin{proof}
Let $p \!\in\! \real^0_t[X]$. Since $\Pi_t(0)=1$, the polynomial $(p-\Pi_t)$ has a root at $0$. Hence, $(p-\Pi_t)(x) = x Q(x)$ with $Q\!\in\!\real_{t-1}[X]$. Then,
\begin{align*}
    G_\nu(p) &= \int p^2(x) x^{-1} \mathrm{d}\nu(x)\\
    & = \int \Pi_t^2(x) x^{-1} \mathrm{d}\nu(x) + 2 \int \Pi_t Q(x) \mathrm{d}\nu(x)\\
    & \quad + \int x Q^2(x) \mathrm{d}\nu(x)\,.   
\end{align*}
The cross-term is equal to $0$ since $Q$ in the span of $\Pi_0, \dots, \Pi_{t-1}$, which are orthogonal to $\Pi_t$. The third term is non-negative, and equal to $0$ if and only if that $Q\!=\!0$. Therefore, the unique solution to~\eqref{EqnPolynomialOptimizationProblem} is $\Pi_t$.
\end{proof}
Based on such an orthogonal family $\{\Pi_t\}$, we aim to derive a first-order method which achieves the lower bound $\mathcal{L}^*_{\mu,t}$. We recall a standard result, that is, for such a family of polynomials $\{\Pi_t\}$, there exist sequences $\{a_t\}$ and $\{b_t\}$ such that $\Pi_0(x)=1$, $\Pi_1(x) = 1+b_1 x$ and for any $t \gre 2$,
\begin{align}
\label{EqnGenericThreeTermsRecursion}
    \Pi_t(x) = (a_t + b_t x) \Pi_{t-1}(x) + (1-a_t) \Pi_{t-2}(x)\,.
\end{align}
Then we can construct an optimal first-order method according to the following result, which is inspired by the work of~\citet{pedregosaorthogonal19} and whose proof is deferred to Appendix~\ref{ProofTheoremGeneralOptimalFirstOrderMethod}. 
\begin{theorem}
\label{TheoremGeneralOptimalFirstOrderMethod}
Given $x_0 \in \real^d$, set $x_1 = x_0 + b_1 H_S^{-1} \nabla f(x_0)$, and for $t \gre 2$,
\begin{align}
\label{EqnGenericOptimalUpdate}
    x_t = x_{t-1} + b_t H_S^{-1} \nabla f(x_{t-1}) + (1-a_t) (x_{t-2}-x_{t-1})\,.
\end{align}
Then, the sequences of iterates $\{x_t\}$ is asymptotically optimal, i.e.,
\begin{align}
    \lim_{n \to \infty} \frac{\Exs{\|\Delta_t\|}^2}{\Exs{\|\Delta_0\|}^2} = \mathcal{L}^*_{\mu,t}\,.
\end{align}
\end{theorem}

Consequently, a strategy to find the optimal first-order method proceeds as follows. First, we characterize the l.s.d.~$\mu$ of the matrix $C_S$, and we find the polynomial $\Pi_t \in \real_t^0[X]$ which achieves the lower bound $\mathcal{L}^*_{\mu,t}$. Then, according to Theorem~\ref{TheoremGeneralOptimalFirstOrderMethod}, we build from the three-terms recursion~\eqref{EqnGenericThreeTermsRecursion} of the orthogonal polynomials $\{\Pi_t\}$ a first-order method which yields an asymptotically optimal sequence of iterates $\{x_t\}$. Our analysis of the Gaussian case is based on standard random matrix theory results, that we recall in details as we leverage them for the analysis of the SRHT case. For the latter, most technicalities actually lie in characterizing the l.s.d.~$\mu$ of $C_S$, and in constructing an orthogonal basis of polynomials for the distribution $\mathrm{d}\nu(x) \!=\! x^{-1} \mathrm{d}\mu\!\left(x^{-1}\right)$.

\subsection{The Gaussian case}
\label{SectionGaussianCase}

Consider an $m \times n$ matrix $S$ with i.i.d.~entries $\mathcal{N}\!\left(0,m^{-1}\right)$. The l.s.d.~of $C_S$ is the Marchenko-Pastur law with density $\mu_{\rho}$ given in~\eqref{EqnMarchenkoPasturDensity}. Denote by $a\!=\!(1\!-\!\sqrt{\rho})^2$ and $b\!=\!(1\!+\!\sqrt{\rho})^2$ the edge eigenvalues. Let $\{\Delta_t\}$ be the sequence of error vectors generated by a first-order method as in~\eqref{EqnFirstOrder}. According to Lemma~\ref{LemmaPolynomialRecursion}, there exists a sequence of polynomials $p_t \in \real_t^0[X]$ such that $\Delta_t = p_t\!\left(C_S^{-1}\right) \Delta_0$, and
\begin{align}
\label{EqnMPLimit}
    \lim_{n \to \infty} \, \frac{\Exs{\|\Delta_t\|}^2}{\Exs{\|\Delta_0\|}^2} = \int_{a}^{b} p_t^2(\lambda^{-1}) \mu_\rho(\lambda) \mathrm{d}\lambda\,,
\end{align}
\begin{lemma}
\label{LemmaGenericConvergenceRateGaussian}
Under the above assumptions and notations, and setting $P_t(x) \!=\! p_t\!\left(\frac{x}{(1-\rho)^2}\right)$, we have
\begin{align}
    \lim_{n \to \infty} \, \frac{\Exs{\|\Delta_t\|}^2}{\Exs{\|\Delta_0\|}^2} \,=\, (1-\rho) \int_a^b P_t^2(x) \frac{1}{x}\mu_{\rho}(x) \, \mathrm{d}x\,.
\end{align}
Consequently, if $\{\Pi_t\}$ is an orthogonal basis of polynomials with respect to $\mu_{\rho}$ such that $\mathrm{deg}(\Pi_t)=t$ and $\Pi_t(0)=1$ then $\overline{\Pi}_t(x) \defn \Pi_t\!\left((1-\rho)^2 x\right)$ achieves the lower bound $\mathcal{L}_{\mu_\rho,t}^*$.
\end{lemma}
\begin{proof}
Using the change of variable $x=(1-\rho)^2/\lambda$, a simple calculation yields that $p^2_t(\lambda^{-1}) \mu_\rho(\lambda)\mathrm{d}\lambda = (1-\rho) P_t^2(x) \frac{1}{x} \mu_\rho(x)\mathrm{d}x$. Applying Lemma~\ref{LemmaOptimalSolutionPolynomialOptimizationProblem} with $\nu=\mu_\rho$, we get that the optimal polynomial $P_t$ is equal to $\Pi_t$, and thus, $p_t$ is exactly $\overline{\Pi}_t(x)$.
\end{proof} 
The Marchenko-Pastur law $\mu_\rho$ is well-studied, and such a construction of polynomials is classical. In this section, we provide a definition by recursion, which is enough to state the optimal algorithm. However, for the proof of the next results, we will consider an alternative construction, from which we establish several intermediate properties useful to the analysis. Define $\Pi_0(x)=1$, $\Pi_1(x) = 1-x$, and for $t \gre 2$,
\begin{align}
\label{EqnRecurrenceOrthogonalPolynomialsGaussian}
    \Pi_t(x) = (1+\rho-x) \Pi_{t-1}(x) - \rho\,\Pi_{t-2}(x)\,.
\end{align}
\begin{lemma}
\label{LemmaOrthogonalFamilyGaussian}
The family of polynomials $\{\Pi_t\}$ is orthogonal with respect to $\mu_\rho$. Further, we have $\Pi_t(0)=1$ and $\mathrm{deg}(\Pi_t)=t$ for all $t \gre 0$.
\end{lemma}
\begin{proof}
We defer the proof to Section~\ref{ProofLemmaOrthogonalFamilyGaussian}.
\end{proof}
Now, set $\overline{\Pi}_t(x) = \Pi_t\!\left((1-\rho)^2x\right)$. From~\eqref{EqnRecurrenceOrthogonalPolynomialsGaussian}, we obtain that $\overline{\Pi}_0(x)=1$, $\overline{\Pi}_1(x)=1-(1-\rho)^2 x$, and for $t \gre 2$,
\begin{align}
\label{EqnRecurrenceScaledOrthogonalPolynomialsGaussian}
    \overline{\Pi}_t(x) = (1\!+\!\rho\!-\!(1\!-\!\rho)^2 x) \overline{\Pi}_{t-1}(x) - \rho\,\overline{\Pi}_{t-2}(x)\,.
\end{align}
According to Lemma~\ref{LemmaGenericConvergenceRateGaussian}, the polynomial $\overline{\Pi}_t$ achieves the lower bound $\mathcal{L}_{\mu_\rho,t}^*$. Further, we identify the recursion formula~\eqref{EqnRecurrenceScaledOrthogonalPolynomialsGaussian} with the three-terms recursion~\eqref{EqnGenericThreeTermsRecursion} by setting $b_t\!=\!-(1-\rho)^2$ for $t \gre 1$, and $a_t\!=\!1+\rho$ for $t \gre 2$. Using Theorem~\ref{TheoremGeneralOptimalFirstOrderMethod}, we immediately have the asymptotically optimal first-order method, which we present in Algorithm~\ref{AlgorithmOptimalFirstOrderGaussian} in its finite-sample approximation. 
\begin{algorithm}[H]
\caption{Optimal First-Order Method for Gaussian embeddings.}
\label{AlgorithmOptimalFirstOrderGaussian}
\begin{algorithmic}
	\STATE {\bfseries Input:} Data matrix $A \in \mathbb{R}^{n \times d}$, sketch size $m \gre d+1$, initial point $x_0 \in \real^d$ and (finite-sample) ratio $\rho \defn d/m$.
	\STATE Sample $S \in \real^{m \times n}$ with i.i.d.~entries $\mathcal{N}(0,1/m)$.
	\STATE Compute the sketched matrix $S_A = S \cdot A$.
	\STATE Compute and cache a factorization of $H_S = S_A^\top S_A$.
	\STATE Set $x_1 = x_0 - (1-\rho)^2 H_S^{-1} \cdot A^\top (Ax_0 - b)$.\\
	\FOR{$t= 2$ {\bfseries to} $T$}
	    \STATE Compute the gradient $g_{t-1} = A^\top (Ax_{t-1} - b)$.
	    \STATE Perform the update
		\begin{align}
		\label{EqnOptimalAlgorithmGaussian}
		    x_t = x_{t-1} + \rho (x_{t-1}-x_{t-2}) - (1-\rho)^2 \cdot H_S^{-1} g_t\,.
		\end{align}
	\ENDFOR
	\STATE Return the last iterate $\mathbf{x_T}$.
\end{algorithmic}
\end{algorithm}
Surprisingly, up to the initialization of the first iterate $x_1$, Algorithm~\ref{AlgorithmOptimalFirstOrderGaussian} corresponds exactly to the Heavy-ball method~\eqref{EqnIHSUpdate} using the fixed step size $\mu=(1-\rho)^2$ and the fixed momentum parameter $\beta=\rho$, which was obtained in~\cite{ozaslan2019iterative, lacotte2019faster} based on edge eigenvalues analysis. Hence, in the Gaussian case, leveraging the whole shape of the limiting distribution, as opposed to using only the edge eigenvalues, yields the same algorithm. We complete the analysis of the Gaussian case by providing the exact asymptotic error $\mathcal{L}^*_{\mu_\rho,t}$.
\begin{theorem}
\label{TheoremOptimalAlgorithmGaussian}
The sequence of iterates $\{x_t\}$ given by Algorithm~\ref{AlgorithmOptimalFirstOrderGaussian} is asymptotically optimal within the class of first order algorithms as in~\eqref{EqnFirstOrder}, and the optimal error is given by $\mathcal{L}_{\mu_\rho,t}^* = \rho^t$.
\end{theorem}
\begin{proof}
We have already argued that $\{x_t\}$ is asymptotically optimal. It remains to show that $\mathcal{L}^*_{\mu_\rho,t} = \rho^t$, whose proof is deferred to Appendix~\ref{ProofTheoremOptimalAlgorithmGaussian}.
\end{proof}

\subsection{The SRHT case}

Haar random projections have been shown to have a better performance than Gaussian embeddings in several contexts. However, they are slow to generate and apply, and we consider instead the SRHT. We recall the definition of the Stieltjes transform $m_\mu$ of a distribution $\mu$ supported on $[0,+\infty)$, which, for $z \in \mathbb{C}\setminus \real_+$, is given by $m_\mu(z) \defn \int_\real \frac{1}{x-z} \, \mathrm{d}\mu(x)$. It has been recently shown that the SRHT behaves asymptotically as Haar embeddings, as formally stated by the next result.
\begin{lemma}[Theorem~4.1 in~\citet{lacotteiterative20}]
\label{LemmaStieltjesHaar}
Let $S$ be an $m \times n$ SRHT embedding and $S_h$ be an $m \times n$ Haar embedding. Then, the matrices $C_S$ and $C_{S_h}$ have the same limiting spectral distribution $F_h$, with support included within the interval $(0,1)$ and whose Stieltjes transform $m_h$ is given by
\begin{align}
\label{EqnStieltjesHaar}
    m_h(z) = \frac{1}{2\gamma}\left(\frac{2\gamma-1}{1-z} + \frac{\xi-\gamma}{z(1-z)} - \frac{R(z)}{z(1-z)}\right)\,,
\end{align}
where
\begin{align*}
    R(z) = \sqrt{(\gamma+\xi-2+z)^2 + 4(z-1)(1-\gamma)(1-\xi)}\,,
\end{align*}
\end{lemma}
\begin{remark}
Due to the computational benefits of the SRHT over Haar projections, we state all our next results for the former, although all statements also apply to the latter (except for the time complexity results). 
\end{remark}
In order to characterize the optimal first-method with SRHT embeddings, we first derive the density of $F_h$.
\begin{theorem}
\label{TheoremHaarDensity}
The distribution $F_h$ admits the following density on $\real$,
\begin{align}
    f_h(x) = \frac{1}{2\gamma\pi} \frac{\sqrt{(\Lambda_h - x)_+(x-\lambda_h)_+}}{x (1-x)}\,,
\end{align}
where
\begin{align*}
\begin{cases}
    \lambda_h \defn \left(\sqrt{(1-\gamma)\xi}-\sqrt{(1-\xi)\gamma}\right)^2\,,\\
    \Lambda_h \defn \left(\sqrt{(1-\gamma)\xi}+\sqrt{(1-\xi)\gamma}\right)^2\,.
\end{cases}
\end{align*}
\end{theorem}
\begin{proof}
The proof is essentially based on the expression~\eqref{EqnStieltjesHaar} of the Stieltjes transform $m_h$, and on the inversion formula,
\begin{align}
\label{EqnInverseStieltjesTransform}
	f_h(x) = \lim_{y \to 0^+} \frac{1}{\pi} \mathrm{Im}\left(m_h(x+iy)\right)\,,\,\, \mathrm{where}\, y \in \real_+\,.
\end{align}
which holds for any $x \in \real$ provided that the above limit exists~\citep{silverstein1995analysis}. We defer the calculations to Appendix~\ref{ProofHaarDensity}.
\end{proof}
Using the change of variable $y=x/\xi$, we can also derive the limiting density of the \emph{rescaled} matrix $\frac{n}{m}C_S$ -- whose expectation is equal to the identity -- which is given by
\begin{align}
    f_{h,r}(y) = \xi f_h(\xi y) = \frac{\sqrt{(\Lambda_{h,r} - y)_+(y-\lambda_{h,r})_+}}{2\rho \pi y (1-\xi y)}\,,
\end{align}
where 
\begin{align*}
\begin{cases}
    \lambda_{h,r} = \lambda_h / \xi = \left(\sqrt{1-\gamma} - \sqrt{(1-\xi) \rho}\right)^2\,,\\
    \Lambda_{h,r} = \Lambda_h / \xi = \left(\sqrt{1-\gamma} + \sqrt{(1-\xi) \rho}\right)^2\,.
\end{cases}
\end{align*}
The density $f_{h,r}$ resembles the Marchenko-Pastur density $\mu_\rho$, up to the factor $(1-\xi y)$ and corrections in the edge eigenvalues $\lambda_{h,r}$ and $\Lambda_{h,r}$. When $\xi, \gamma \approx 0$, then $\lambda_{h,r} \!\approx\! (1-\sqrt{\rho})^2$, $\Lambda_{h,r} \!\approx\! (1+\sqrt{\rho})^2$, and $f_{h,r}(x) \approx \mu_{\rho}(x)$. This is consistent with the fact that provided $m, d = o(n)$ so that $\xi,\gamma=0$, then the l.s.d.~of $\frac{n}{m}C_S$ is the Marchenko-Pastur law with parameter $\rho$ (see~\citep{jiang2009approximation} for a formal statement). In Figure~\ref{FigSRHTDensity}, we compare the empirical spectral density of the matrix $\frac{n}{m} C_S$ with $S$ an $m \times n$ SRHT to $f_{h,r}$, for fixed $d$ and $n$, and several values of $m$. We observe that these two densities match very closely, and so does the empirical spectral density using a Haar projection with $f_{h,r}$. Further, as $m$ increases, the limiting density $f_{h,r}$ departs from $\mu_\rho$, and then concentrates more and more around $1$. Note in particular that the support of $f_{h,r}$ is always within that of $\mu_\rho$. This can be formally verified by comparing their respective edge eigenvalues.

\begin{figure}[h!]
	\centering
	\includegraphics[width=8cm]{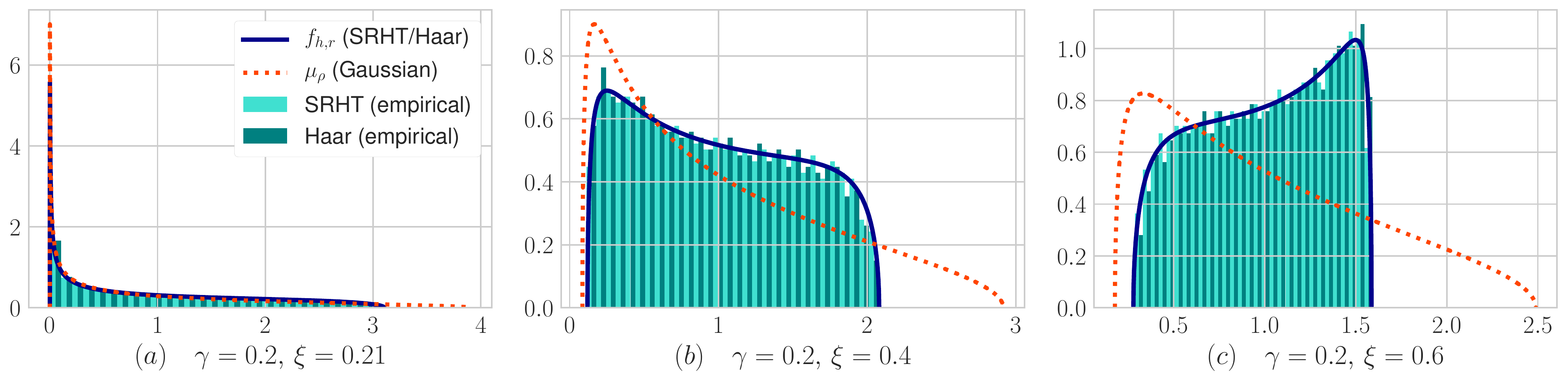}
	\caption{We use $n=8192$, $\gamma \approx \frac{d}{n} = 0.2$ and $\xi \approx \frac{m}{n} \in \{0.21, 0.4, 0.6\}$.}
	\label{FigSRHTDensity}
\end{figure}

\subsubsection{Orthogonal polynomials and optimal first-order method}
\label{sec:orthpoly}
Given a first-order method as in~\eqref{EqnFirstOrder}, we know from Lemma~\ref{LemmaPolynomialRecursion} that for a given iteration $t$, there exists a polynomial $p \in \real^0_t[X]$ such that $\Delta_t \!=\! p\!\left(C_S^{-1}\right) \Delta_0$, and
\begin{align}
\label{EqnErrorSRHT}
    \lim_{n \to \infty} \frac{\Exs{\|\Delta_t\|}^2}{\Exs{\|\Delta_0\|}^2} = \int_{\lambda_h}^{\Lambda_h} p^2\!\left(\lambda^{-1}\right) f_h(\lambda) \, \mathrm{d}\lambda\,.
\end{align}
Introducing the scaling parameters $\tau \!=\! \left(\frac{\sqrt{\Lambda_h}-\sqrt{\lambda_h}}{\sqrt{\Lambda_h}+\sqrt{\lambda_h}}\right)^2$, $c \!=\! \frac{4}{\left(\sqrt{1/\Lambda_h}+\sqrt{1/\lambda_h}\right)^2}$, $\alpha\!=\!(1-\sqrt{\tau})^2$, $\beta\!=\!(1+\sqrt{\tau})^2$, the rescaled polynomial $P(x) = p(x/c)$, and using the change of variable $x\!=\!c/\lambda$, we find that 
\begin{align}
    & \lim_{n \to \infty} \frac{\Exs{\|\Delta_t\|}^2}{\Exs{\|\Delta_0\|}^2}\\ =\,& \frac{c \tau}{(1-\tau)\gamma} \int_\alpha^\beta P^2(x) \frac{\sqrt{(x-\alpha)(\beta-x)}}{2 \pi \tau x (x-c)} \, \mathrm{d}x\\
    = \, & \frac{c \tau}{(1-\tau)\gamma} \int_\alpha^\beta P^2(x) \frac{\mu_\tau(x)}{x-c} \, \mathrm{d}x \label{EqnMinimizationProblemSRHT}
\end{align}
Thus, according to Lemma~\ref{LemmaOptimalSolutionPolynomialOptimizationProblem}, it suffices to find a family of polynomials $\{R_t\}$ orthogonal with respect to the density $\frac{x \mu_\tau(x)}{x-c}$ such that $\mathrm{deg}(R_t)=t$ and $R_t(0)=1$, in which case the minimizer over $P \in \real_t^0[X]$ of the integral in~\eqref{EqnMinimizationProblemSRHT} is equal to $R_t$, and the minimizer of~\eqref{EqnErrorSRHT} is then $\overline{R}_t(x) \!=\! R_t(cx)$.
\begin{theorem}
\label{thm:orthpoly}
Define the parameters $\omega\!=\!\frac{4}{\left(\sqrt{\beta-c}+\sqrt{\alpha-c}\right)^2}$ and $\kappa\!=\!\left(\frac{\sqrt{\beta-c}-\sqrt{\alpha-c}}{\sqrt{\beta-c}+\sqrt{\alpha-c}}\right)^2$. Let $\{\Pi_t\}$ be the orthogonal family of polynomials with respect to $\mu_\kappa$, that is, $\Pi_0(x)=1$, $\Pi_1(x)=1-x$, and for $t \gre 2$,
\begin{align}
\label{EqnRecursion22}
    \Pi_t(x) = (1+\kappa-x) \Pi_{t-1}(x) - \kappa \Pi_{t-2}(x)\,.     
\end{align}
Define the polynomials $R_t(x) = \Pi_t(\omega(x-c)) / \Pi_t(-\omega c)$. Then, it holds that $R_t(0)=1$, $\mathrm{deg}(R_t)=t$, and the family $\{R_t\}$ is orthogonal with respect to the density $\frac{x \mu_\tau(x)}{x-c}$.
\end{theorem}
\begin{proof}
For $k \!\neq \!\ell$, we have that $\int_\real R_k(x) R_\ell(x) \frac{x\mu_\tau(x)}{x-c} \,\mathrm{d}x \propto \int_\alpha^\beta \Pi_k\left(\omega(x-c)\right) \Pi_\ell\left((\omega(x-c)\right) \frac{\sqrt{(\beta-x)(x-\alpha)}}{2 \pi \rho (x-c)} \,\mathrm{d}x$. Using the change of variable $y=\omega(x-c)$, we find that the latter integral is (up to a constant) equal to $\int_\real \Pi_k(y) \Pi_\ell(y) \mu_\kappa(y) \, \mathrm{d}y$, which is itself equal to $0$ due to the orthogonality of the $\Pi_t$ with respect to $\mu_\kappa$.
\end{proof}
In order to derive the optimal first-order method, we need to find the three-terms recursion relationship satisfied by the polynomials $\{\overline R_t\}$. First, let us compute the normalization factor $u_t \!\defn\! \Pi_t(-\omega c)$. Evaluating~\eqref{EqnRecursion22} at $x=-\omega c$ and denoting $\eta \defn 1+\kappa+\omega c$, we find that $u_{t+1} = \eta u_t - \kappa u_{t-1}$, with the initial conditions $u_0=1$ and $u_1 = \Pi_1(-\omega c) = 1+\omega c = \eta -\kappa$. Thus, after solving this second-order linear system, we obtain that
\begin{align}
    & u_t = \frac{x_1-\kappa}{x_1 - x_2} x_1^t + \frac{ \kappa-x_2}{x_1-x_2} x_2^t\,,
\end{align}
where $x_1 \!=\! \frac{\eta}{2} \!+\! \sqrt{\frac{\eta^2}{4}-\kappa}$ and $x_2 \!=\! \frac{\eta}{2} \!-\! \sqrt{\frac{\eta^2}{4}-\kappa}$. It is easy to check that $\eta^2/4 > \kappa$, so that $x_1$ and $x_2$ are indeed distinct and real. Then, using the change of variable $y=\omega (x-c)$ in~\eqref{EqnRecursion22}, we get the following three-terms recurrence relationship, that is, $\overline R_0(x) = 1$, $\overline R_1(x) = 1 + b_{h,1} x$ and for $k \gre 2$,
\begin{align}
    \overline R_t(x) = (a_{h,t} + x b_{h,t}) \overline R_{t-1}(x) + (1-a_{h,t}) \overline R_{t-2}(x)\,,
\end{align}
where $a_{h,t} = \frac{\eta\, u_{t-1}}{u_t}$ for $t \gre 1$, and $b_{h,t} = -\frac{\omega c \, u_{t-1}}{u_t}$ for $t \gre 2$. Using Theorem~\ref{TheoremGeneralOptimalFirstOrderMethod}, we obtain the optimal first-order method, which we present in Algorithm~\ref{AlgorithmOptimalFirstOrderHaar} in its finite-sample approximation.
\begin{algorithm}[H]
    \caption{Optimal First-Order Method for SRHT (or Haar) embeddings.}
    \label{AlgorithmOptimalFirstOrderHaar}
\begin{algorithmic}
	\STATE {\bfseries Input:} Data matrix $A \in \mathbb{R}^{n \times d}$, sketch size $m \gre d+1$, initial point $x_0 \in \real^d$.
	\STATE Sample an $m \times n$ SRHT $S$.
	\STATE Compute the sketched matrix $S_A = S \cdot A$.
	\STATE Compute and cache a factorization of $H_S = S_A^\top S_A$.
	\STATE Set $x_1 = x_0 + b_{h,1} H_S^{-1} A^\top (A x_0 - b)$.
	\FOR{$t= 2$ {\bfseries to} $T$}
	    \STATE Compute the gradient $g_{t-1} = A^\top (Ax_{t-1} - b)$. 
	    \STATE Perform the update
		\begin{align}
		\label{EqnOptimalAlgorithmHaar}
		    x_t = x_{t-1} + b_{h,t} H_S^{-1} g_t + (1-a_{h,t}) (x_{t-2}-x_{t-1})\,.
		\end{align}
		where $a_{h,t}$ and $b_{h,t}$ are as described in Section \ref{sec:orthpoly}.
	\ENDFOR
	\STATE Return the last iterate $\mathbf{x_T}$.
\end{algorithmic}
\end{algorithm}
Differently from the Gaussian case, Algorithm~\ref{AlgorithmOptimalFirstOrderHaar} does not correspond to the Heavy-ball method~\eqref{EqnIHSUpdate} using the fixed step size $\mu\!=\!(1-\rho)^2$ and the fixed momentum parameter $\beta \!=\! \rho$, which was obtained by~\cite{lacotte2019faster} based on edge eigenvalues analysis and standard finite-sample concentration bounds on the spectrum of SRHT matrices~\citep{tropp2011improved}. 

Using the new asymptotically exact extreme eigenvalues we derived in Theorem~\ref{TheoremHaarDensity} -- which are different from the bounds obtained by~\cite{tropp2011improved} -- and following the same extreme eigenvalues analysis proposed by~\cite{lacotte2019faster}, we can derive an optimal Heavy-ball method for which the step size $\mu_h$ and momentum parameter $\beta_h$ are given by $\mu_h = \frac{4}{\left(\frac{1}{\sqrt{\Lambda_h}}+\frac{1}{\sqrt{\lambda_h}}\right)^2}$ and $\beta_h = \left(\frac{\sqrt{\Lambda_h}-\sqrt{\lambda_h}}{\sqrt{\Lambda_h} + \sqrt{\lambda_h}}\right)^2$. 

Hence, leveraging the whole shape of the limiting distribution, as opposed to using only the edge eigenvalues, yields an optimal first-order method which is different, and has non-constant step sizes and momentum parameters. But interestingly, it holds that as the iteration number $t$ grows to $+\infty$, then the update coefficients $a_{h,t}$ and $b_{h,t}$ have respective limits $1+\beta_h$ and $-\mu_h$, which yields exactly this Heavy-ball method. Thus, we expect the latter and Algorithm~\ref{AlgorithmOptimalFirstOrderHaar} to have a similar performance as $t$ grows large.

We complete our analysis of the SRHT case by characterizing the asymptotic error $\mathcal{L}^*_{f_h,t}$.
\begin{theorem}
\label{TheoremOptimalAlgorithmHaar}
The sequence of iterates $\{x_t\}$ given by Algorithm~\ref{AlgorithmOptimalFirstOrderHaar} is asymptotically optimal, and the optimal error satisfies $\mathcal{L}_{f_h,t}^* \asymp \frac{(1-\xi)^t}{(1-\gamma)^t} \, \rho^t$.
\end{theorem}
\begin{proof}
We have already argued that $\{x_t\}$ is asymptotically optimal. It remains to show that $\mathcal{L}^*_{f_h,t} \asymp \frac{(1-\xi)^t}{(1-\gamma)^t} \, \rho^t$, whose proof is deferred to Appendix~\ref{ProofTheoremOptimalAlgorithmHaar}.
\end{proof}
Of natural interest is to compare the rate of convergence $\rho_h \!\defn\! \frac{(1-\xi)}{(1-\gamma)} \rho$ of Algorithm~\ref{AlgorithmOptimalFirstOrderHaar} to the rate $\rho$ of Algorithm~\ref{AlgorithmOptimalFirstOrderGaussian}. We have $\frac{\rho_h}{\rho} = \frac{(1-\xi)}{(1-\gamma)}$, which is always smaller than $1$ since $\xi > \gamma$. Hence, these rotation matrices yield an optimal first-order method which is uniformly better than that with Gaussian embeddings, by a factor which can be made arbitrarily large by increasing the sketch size $m$ relatively to the other dimensions. Further, if we do not reduce the size of the original matrix, so that $m=n$ and $\xi=1$, then the algorithm converges in one iteration. This means that we do not lose any information by sketching. In contrast, Gaussian projections introduce more distortions than rotations, even though the rows of a Gaussian matrix are almost orthogonal to each other in the high-dimensional setting.

Further, we compare the rate of Algorithm~\ref{AlgorithmOptimalFirstOrderHaar} to the rate of the best Heavy-ball method with refreshed SRHT embeddings which is equal to $\rho_h^\text{ref} = \rho \cdot \frac{\xi(1-\xi)}{\gamma^2 + \xi - 2\xi \gamma}$. We have $\rho_h < \rho_h^\text{ref}$ if and only if $\frac{1-\xi}{1-\gamma} < \frac{\xi (1-\xi)}{\gamma^2+\xi-2\xi\gamma}$, which is equivalent to $\gamma^2 + \xi - 2\gamma \xi < \xi - \gamma \xi$, again equivalent to $\gamma^2 < \gamma \xi$, i.e., $\gamma < \xi$, which holds by assumption. Thus, a fixed embedding yields a first-order method which is uniformly faster than the best Heavy-ball method with refreshed sketches. However, it remains an open problem whether one can find a first-order method with refreshed sketches which yields a rate better than $\rho_h^\text{ref}$. We recapitulate the different convergence rates in Table~\ref{TabComparisonConvergenceRates}.
\begin{table}[!h]
	\caption{Asymptotic rates of convergence for the best first-order method~\eqref{EqnFirstOrder} and the best Heavy-ball method~\eqref{EqnIHSUpdate}, with fixed or refreshed Gaussian or SRHT embeddings. For the best Heavy-ball method rates, we use previously derived results from~\cite{ozaslan2019iterative, lacotte2019faster, lacotteiterative20}.}
	\label{TabComparisonConvergenceRates}
	\centering
	\begin{scriptsize}
	\begin{tabular}{|c|c|c|c|c|}
		\toprule
		Algorithm & Fixed & Refreshed & Fixed & Refreshed \\
		& Gaussian & Gaussian & SRHT & SRHT\\
		\midrule
		Best first-order& $\rho$ & unknown & $\frac{1-\xi}{1-\gamma} \rho$ & unknown \\
	    method~\eqref{EqnFirstOrder}  & & & & \\
		\midrule
		Best Heavy-ball & $\rho$ & $\rho$ & $\rho$ & $\frac{\xi(1-\xi)}{\gamma^2 + \xi - 2\xi \gamma}$ \\
		method~\eqref{EqnIHSUpdate} & & & & \\
		\bottomrule
	\end{tabular}
	\end{scriptsize}
\end{table}

In Figure~\ref{FigComparisonIHS}, we verify numerically that Algorithm~\ref{AlgorithmOptimalFirstOrderHaar} is faster than the best Heavy-ball method with refreshed SRHT sketches ("SRHT (refreshed)"), and than Algorithm~\ref{AlgorithmOptimalFirstOrderGaussian}. Further, we compare Algorithm~\ref{AlgorithmOptimalFirstOrderHaar} to the Heavy-ball method with fixed SRHT embedding whose parameters are found based on edge eigenvalues analysis, using either our new density $f_h$ ("SRHT (edge eig.)") -- as described previously in Section~\ref{sec:orthpoly} --, or, the previous bounds derived by~\cite{tropp2011improved} ("SRHT (baseline)"). As predicted, Algorithm~\ref{AlgorithmOptimalFirstOrderHaar} performs very similarly to the former, and better than the latter. Finally, we verify that our predicted convergence rates for Algorithms~\ref{AlgorithmOptimalFirstOrderGaussian} and~\ref{AlgorithmOptimalFirstOrderHaar} are matched empirically, on Figure~\ref{FigComparisonRates}. We mention that we use small perturbations of the algorithmic parameters derived from our asymptotic analysis. Following the notations introduced in Theorem~\ref{TheoremGeneralOptimalFirstOrderMethod}, instead of $a_t$ and $b_t$, we use $a_t^\delta=(1+\delta)a_t$ and $b_t^\delta = (1-\delta)b_t$ with $\delta=0.01$. These conservative perturbations are necessary in practice due to the finite-sample approximations. We defer a detailed description of the experimental setup to Appendix~\ref{SectionExperimentalSetup}.
\begin{figure}[h!]
	\centering
	\includegraphics[width=8cm]{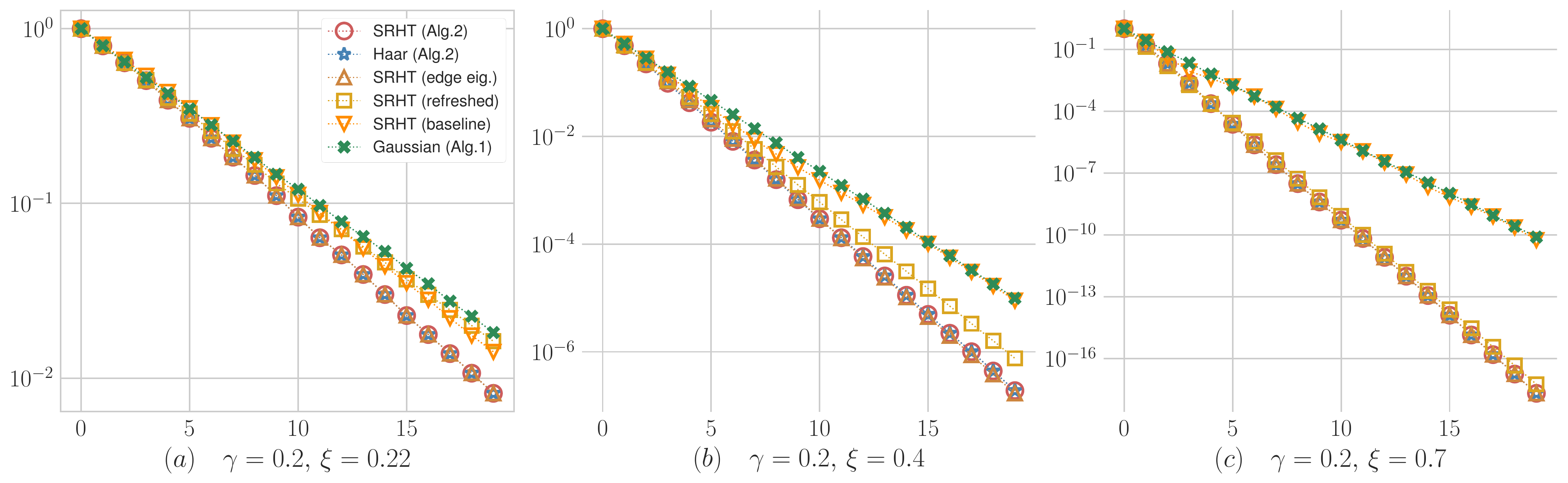}
	\caption{Error $\Exs \|\Delta_t\|^2 / \Exs \|\Delta_0\|^2$ versus number of iterations. We use $n = 8192$, $d/n \approx \gamma = 0.2$ and $m/n \approx \xi \in \{0.22, 0.4, 0.7\}$.}
	\label{FigComparisonIHS}
\end{figure}
\begin{figure}[h!]
	\centering
	\includegraphics[width=8cm]{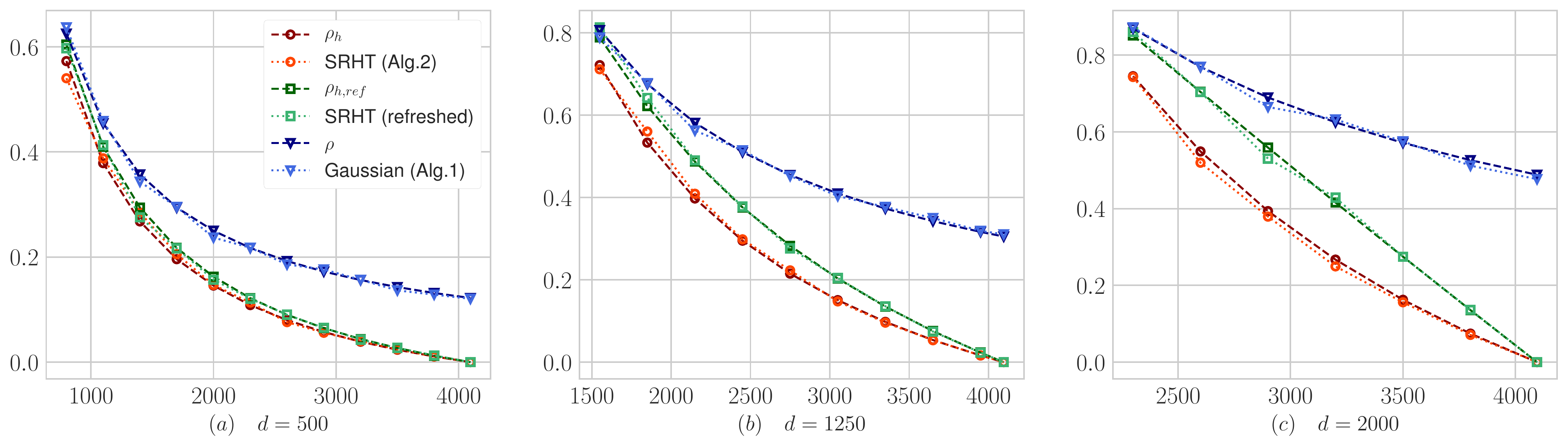}
	\caption{Empirical and theoretical convergence rates versus sketch size $m$. We use $n = 8192$ and $d \in \{500, 1250, 2000\}$.}
	\label{FigComparisonRates}
\end{figure}

\subsubsection{Complexity Analysis}

We turn to a complexity analysis of Algorithm~\ref{AlgorithmOptimalFirstOrderHaar} and compare it to the currently best known algorithmic complexities for solving~\eqref{EqnMain}. 

Given a fixed (and independent of the dimensions) error $\varepsilon \!>\! 0$, we aim to find $\widetilde x$ such that $\|A(\widetilde x - x^*)\|^2 \less \varepsilon$. Among the best complexity algorithms is the pre-conditioned conjugate gradient algorithm~\citep{rokhlin2008fast}. As described in Section~\ref{SectionIntroduction}, it is decomposed into three parts: sketching the data matrix, factoring the pre-conditioned matrix, and then the iterations of the conjugate gradient method. This algorithm prescribes at least the sketch size $m \asymp d \log d$ in order to converge with high-probability guarantees. This theoretical prescription is based on the finite-sample bounds on the extremal eigenvalues of the matrix $C_S$ derived by~\cite{tropp2011improved}. Then, the resulting complexity scales as
\begin{align}
    \mathcal{C}_{\text{cg}} \asymp nd \log d + d^3 \log d + nd \log(1/\varepsilon)\,,
\end{align}
where $nd \log d$ is the sketching cost, $d^3 \log d$ the pre-conditioning cost, and $nd \log(1/\varepsilon)$ is the per-iteration cost $nd$ times the number of iterations $\log 1/\varepsilon$. 

Our analysis shows that for $m \asymp d$, Algorithm~\ref{AlgorithmOptimalFirstOrderHaar} yields a complexity no larger than
\begin{align}
    \mathcal{C}_{\text{fhs}} \asymp nd \log d + d^3 + nd \log(1/\varepsilon)\,,
\end{align}
Note that in the above complexity, we omit the rate of convergence -- which would yield an even smaller complexity -- to simplify the comparison. Since $\varepsilon$ is independent of the dimensions, it follows that 
\begin{align}
    \frac{C_{\text{fhs}}}{C_{\text{cg}}} \asymp \frac{1}{\log d}\,, \quad d \to \infty\,.
\end{align}
Hence, with a smaller sketch size, the resulting complexity improves by a factor $\log d$ over the current state-of-the-art in randomized preconditioning for dense problems (e.g., see \cite{boutsidis2013improved,nelson2013osnap}). We also note that the $O(d^3)$ term can be improved to $O(d^\omega)$, where $\omega$ is the exponent of matrix multiplication. 

It has also been shown by~\citet{lacotteiterative20} that the Heavy-ball update~\eqref{EqnIHSUpdate} with refreshed SRHT embeddings yields a complexity $\mathcal{C}_{\text{ihs}}$ such that $C_{\text{ihs}}/C_{\text{cg}} \asymp \frac{1}{\log d}$, provided that $m \asymp d$. In order to compare more finely Algorithm~\ref{AlgorithmOptimalFirstOrderHaar} with this algorithm, we consider an arbitrary sketch size $m$. Then, the complexity of Algorithm~\ref{AlgorithmOptimalFirstOrderHaar} is
\begin{align}
    \mathcal{C}_{\text{fhs}} \asymp nd \log m + md^2 + nd \frac{\log(1/\varepsilon)}{\log \rho_h}\,, 
\end{align}
whereas the former algorithm yields
\begin{align}
    \mathcal{C}_{\text{ihs}} \asymp \left(nd \log m + md^2 + nd \right) \frac{\log(1/\varepsilon)}{\log \rho_h^\text{ref}}\,.
\end{align}
Since, in particular, $\rho_h$ is uniformly smaller than $\rho_h^\text{ref}$, it always holds that 
\begin{align}
    \mathcal{C}_{\text{fhs}} \less \mathcal{C}_{\text{ihs}}\,.
\end{align}
It should be noted that we translate our asymptotic results to finite-sample versions. Although it is beyond our scope, we believe that our results could be extended to finite-sample versions with high-probability guarantees and with similar rates of convergence.

\section*{Acknowledgements}
The authors thank Sifan Liu and Edgar Dobriban for helpful discussions. This work was partially supported by the National Science Foundation under grant IIS-1838179.



\newpage 

\onecolumn
\appendix

\section{Proof of main results}

For a polynomial $P$ and a measure (resp.~density) $\mu$, we will denote $\mu[P] \defn \int_\real P(x) \mu(x) \mathrm{d}\mu(x)$ (resp.~$\mu[P] \defn \int_\real P(x)\mu(x) \mathrm{d}x$). For a density $\mu$, we stress the fact that $\mu[x]$ and $\mu(x)$ refer to different quantities.

\subsection{Proof of Theorem~\ref{TheoremGeneralOptimalFirstOrderMethod}}
\label{ProofTheoremGeneralOptimalFirstOrderMethod}

We recall that $\Pi_0(x) \!=\! 1$, $\Pi_1(x) \!=\! 1 + b_1 x$ and for $t \gre 2$,
\begin{align}
\label{EqnRescaledPolyRec2}
    \Pi_t(x) = \left(a_t + b_t x\right) \Pi_{t-1}(x) + (1-a_t) \Pi_{t-2}(x)\,.
\end{align}
First, we claim that for any $t \gre 0$, $\Delta_t \!=\! \Pi_t\!\left(C_S^{-1}\right) \Delta_0$, and we show it by induction. Since $\Pi_0(x)=1$, we have that $\Delta_0 = \Pi_0\!\left(C_S^{-1}\right) \cdot \Delta_0$. Since $x_1 \!=\! x_0 + b_1 H_S^{-1} \nabla f(x_0)$, subtracting $x^*$ and multiplying by $U^\top A$ the latter equation, we obtain that $\Delta_1 = \Pi_1\!\left(C_S^{-1}\right) \cdot \Delta_0$. Suppose that for some $t \gre 2$, the induction claim holds for $t-1$ and $t-2$. Subtracting $x^*$ and multiplying by $U^\top A$ the update formula~\eqref{EqnGenericOptimalUpdate}, we obtain that
\begin{align*}
    \Delta_t &= \Delta_{t-1} + (1-a_t) (\Delta_{t-2} - \Delta_{t-1}) + b_t C_S^{-1} \, \Delta_{t-1}\\
    &= (a_t + b_t C_S^{-1}) \Delta_{t-1} + (1-a_t) \Delta_{t-2}\\
    &= \left((a_t + b_t C_S^{-1}) \Pi_{t-1}\!\left(C_S^{-1}\right) + (1-a_t) \Pi_{t-2}\!\left(C_S^{-1}\right) \right) \Delta_0\\
    &= \Pi_t\!\left(C_S^{-1}\right) \Delta_0\,,
\end{align*}
where we used the induction hypothesis for $t\!-\!1$ and $t\!-\!2$ in the third equality, and the recursion formula~\eqref{EqnRescaledPolyRec2} in the last equality. Consequently, using Lemma~\ref{LemmaPolynomialRecursion}, we obtain that
\begin{align*}
    \lim_{n \to \infty} \, \frac{\Exs {\|\Delta_t\|}^2}{\Exs{\|\Delta_0\|}^2} = \int_a^b \Pi_t^2\!\left( \lambda^{-1}\right) \mathrm{d} \mu(\lambda) = \mathcal{L}^*_{\mu,t}\,.
\end{align*}

\subsection{Proof of Theorem~\ref{TheoremOptimalAlgorithmGaussian}}
\label{ProofTheoremOptimalAlgorithmGaussian}
We have already argued that $\{x_t\}$ is asymptotically optimal. It remains to prove that $\mathcal{L}_{\mu_\rho,t}^* \!=\! \rho^t$. 

Set $\lambda_\rho(x)=x^{-1} \mu_\rho(x)$. Let $\{\Pi_t\}$ be an orthogonal basis with respect to $\mu_\rho$ such that $\Pi_t(0)=1$ and $\mathrm{deg}(\Pi_t)=t$. From Lemma~\ref{LemmaGenericConvergenceRateGaussian}, we have $\mathcal{L}_{\mu_\rho,t}^* = (1-\rho) \lambda_\rho[\Pi_t^2]$, so that it suffices to show that $\lambda_\rho[\Pi_t^2] = (1-\rho)^{-1} \rho^{t}$. On the other hand, in the proof of Lemma~\ref{LemmaOrthogonalFamilyGaussian} in Appendix~\ref{ProofLemmaOrthogonalFamilyGaussian}, we establish that there exists a sequence of polynomials $\{T_k\}_{k \gre 1}$ such that for any $t \gre 1$ and $k, \ell \gre 1$,
\begin{align*}
    & \Pi_t(x) = 1 - \sum_{j=1}^t \lambda_\rho[T_t] \, T_t(x)\,,\\
    & \lambda_\rho[T_t] = (-1)^{t-1} \sqrt{\rho}^{t-1}\,,\\
    & \lambda_\rho[T_k T_\ell] = \delta_{k \ell}\,,
\end{align*}
where $\delta_{k\ell} = 1$ if $k=\ell$, and $0$ otherwise. Using the latter properties, it follows that
\begin{align*}
    \lambda_\rho[\Pi_t^2] &= \lambda_\rho[1] - 2 \, \sum_{j=1}^t \lambda_\rho[T_j]^2 + \sum_{j=1}^t \lambda_\rho[T_j]^2 \underbrace{\lambda_\rho[T_j^2]}_{=1}\\
    &= \lambda_\rho[1] - \sum_{j=1}^t \lambda_\rho[T_j]^2\\
    &= \frac{1}{1-\rho} - \sum_{j=0}^{t-1} \rho^j\\
    &= \frac{\rho^{t}}{1-\rho}\,,
\end{align*}
and, in the third equality, we used the standard inverse moment formula $\lambda_\rho[1]\!=\!\int x^{-1} \mu_\rho(x)\mathrm{dx}\!=\!(1-\rho)^{-1}$. Consequently, we obtain the claimed formula, that is, $\mathcal{L}^*_{\mu_\rho,t}=\rho^t$.

\subsection{Proof of Theorem~\ref{TheoremHaarDensity}}
\label{ProofHaarDensity}

According to Lemma~\ref{LemmaStieltjesHaar}, the support of $F_h$ is included within the interval $(0,1)$. Therefore, we fix $x \in (0,1)$ and we consider the complex number $z=x+iy$, where $y > 0$. Our goal is to compute the quantity
\begin{align*}
    \lim_{y\to0^+} \frac{1}{\pi} |\mathrm{Im}\!\left(m_h(z)\right)|\,.
\end{align*}
If the above limit exists, then $F_h$ is differentiable at $x$ and its derivative is equal to this limit~\citep{silverstein1995analysis}. Note that the absolute value is not necessary, since $\mathrm{Im}\!\left(m_h(z)\right)$ is positive on $\mathbb{C}^+$. But it will avoid to specify explicitly the branch cut of the square-root considered later in this proof, and thus additional technicalities.

From Lemma~\ref{LemmaStieltjesHaar}, it holds that
\begin{align}
\label{EqnDecompositionmh}
    2 \gamma m_h(z) = \frac{2\gamma-1}{1-z} + \frac{\xi-\gamma}{z(1-z)} - \frac{R(z)}{z(1-z)}\,.
\end{align}
where $R(z) = \sqrt{(\gamma+\xi-2+z)^2 + 4(z-1)(1-\gamma)(1-\xi)}$, and the branch cut of the square-root is chosen such that $m_h > 0$ on $\mathbb{C}^+$, $m_h < 0$ on $\mathbb{C}^-$ (the complex numbers with negative imaginary parts), and $m_h > 0$ on $\real_-$ (the negative real numbers). Further, we have
\begin{align*}
    \frac{1}{z(1-z)} = \frac{x(1-x)+y^2 + i y (2x-1)}{(x(1-x)+y^2)^2 + y^2 (2x-1)^2}\,,\qquad \frac{1}{1-z} = \frac{1-x+iy}{(1-x)^2 + y^2}\,,
\end{align*}
from which we deduce that the imaginary parts of the first two terms in the expansion~\eqref{EqnDecompositionmh} of $2 \gamma m_h(z)$ are given by
\begin{align*}
    & \mathrm{Im}\!\left(\frac{2\gamma-1}{1-z}\right) = \frac{(2\gamma-1) y}{(1-x)^2 + y^2}\,,\\
    & \mathrm{Im}\!\left(\frac{\xi-\gamma}{z(1-z)}\right) = \frac{(\xi-\gamma)(2x-1)y}{(x(1-x)+y^2)^2 + y^2 (2x-1)^2}\,.
\end{align*}
Since $x \in (0,1)$, the limits $y \to 0^+$ of the two above quantities exist and are equal to $0$. Hence, provided it exists, we have
\begin{align}
\label{EqnLimit11}
    \lim_{y \to 0^+} 2 \gamma |\mathrm{Im}(m_h(z))| &= \lim_{y \to 0^+} \left|\mathrm{Im}\!\left(\frac{R(z)}{z(1-z)}\right)\right|\,.
\end{align}
We introduce the function $f(z) \!=\! (z-\alpha-\beta)^2 + 4(z-1)\alpha \beta$ where $\alpha = 1-\xi$ and $\beta=1-\gamma$, so that $R(z)\!=\!\sqrt{f(z)}$. We have $f(z) = X + i Y$ where
\begin{align*}
    & X = (x-\alpha-\beta)^2 - y^2 + 4(x-1)\alpha \beta\,, \\
    & Y = 2(x-\alpha - \beta + 2\alpha \beta) y\,.
\end{align*}
Thus, the absolute values of the real and imaginary parts of $R(z)$ are given by
\begin{align*}
    & |\mathrm{Re}(R(z))| = \frac{1}{\sqrt{2}} \sqrt{\sqrt{X^2 + Y^2} + X}\,,\\
    & |\mathrm{Im}(R(z))| = \frac{1}{\sqrt{2}} \sqrt{\sqrt{X^2 + Y^2} - X}\,,
\end{align*}
and they have respective limits
\begin{align*}
    & \lim_{y \to 0^+} |\mathrm{Re}(R(z))| = \sqrt{|\varphi(x)|} \cdot \mathbf{1}(\varphi(x) > 0)\,,\\
    & \lim_{y \to 0^+} |\mathrm{Im}(R(z))| = \sqrt{|\varphi(x)|} \cdot \mathbf{1}(\varphi(x) < 0)\,,
\end{align*}
where $\varphi(x) \defn (x-\alpha-\beta)^2 + 4(x-1)\alpha \beta$. Further, we have
\begin{align*}
    \mathrm{Im}\left(\frac{R(z)}{z(1-z)}\right) = \frac{y(2x-1)\mathrm{Re}(R(z)) + (x(1-x)+y^2) \mathrm{Im}(R(z))}{g(x,y)}\,,
\end{align*}
where $g(x,y) = (x(1-x)+y^2)^2 + y^2 (2x-1)^2$. Note that $\lim_{y \to 0^+} g(x,y) = x^2 (1-x)^2$, which is non-zero since $x \in (0,1)$. Then, we obtain
\begin{align*}
    \lim_{y \to 0^+} \left|\mathrm{Im}\!\left(\frac{R(z)}{z(z-1)}\right)\right| = \frac{\sqrt{|\varphi(x)|} \mathbf{1}(\varphi(x) < 0)}{x(1-x)}\,,
\end{align*}
Using~\eqref{EqnLimit11}, it follows that for any $x \in (0,1)$, $\lim_{y \to 0^+} \frac{1}{\pi} |\mathrm{Im}(m_h(x))|$ exists. This implies that $F_h$ admits a density over $(0,1)$, given by
\begin{align*}
    f_h(x) = \frac{1}{2\gamma\pi} \frac{\sqrt{(\Lambda_h - x)_+(x-\lambda_h)_+}}{x (1-x)}\,,
\end{align*}
where we used the fact that $\varphi(x) = (x-\Lambda_h)(x-\lambda_h)$, and we recall that the edge eigenvalues $\Lambda_h$ and $\lambda_h$ are given by
\begin{align*}
    & \lambda_h \defn \left(\sqrt{(1-\gamma)\xi}-\sqrt{(1-\xi)\gamma}\right)^2\\
    & \Lambda_h \defn \left(\sqrt{(1-\gamma)\xi}+\sqrt{(1-\xi)\gamma}\right)^2\,.
\end{align*}
Using the fact that $F_h$ is supported within the interval $(0,1)$, we have recovered the whole density of the limiting spectral distribution $F_h$ of the matrix $U^\top S^\top S U$.

\subsection{Proof of Theorem~\ref{TheoremOptimalAlgorithmHaar}}
\label{ProofTheoremOptimalAlgorithmHaar}

We have already argued that $\{x_t\}$ is asymptotically optimal. It remains to show that $\mathcal{L}_{f_h,t}^* \asymp \frac{(1-\xi)^t}{(1-\gamma)^t} \rho^t$.

Using~\eqref{EqnMinimizationProblemSRHT}, we have
\begin{align*}
    \mathcal{L}_{f_h,t}^* &= \frac{c \tau}{(1-\tau)\gamma} \min_{P \in \real_t^0[X]} \int_\alpha^\beta P^2(t) \frac{\mu_\tau(x)}{x-c} \, \mathrm{d}x\\
    &= \frac{c \tau}{(1-\tau)\gamma} \min_{P \in \real_t^0[X]} \int_\alpha^\beta P^2(t) \frac{x}{x-c} \frac{\mu_\tau(x)}{x} \, \mathrm{d}x\,.
\end{align*}
For $x \in [\alpha, \beta]$, it holds that
\begin{align*}
    \frac{\beta}{\beta - c} \less \frac{x}{x-c} \less \frac{\alpha}{\alpha-c}\,,
\end{align*}
and consequently, we can lower and upper bound $\mathcal{L}_{f_h,t}^*$ as follows, 
\begin{align*}
    \frac{c \tau}{(1-\tau)\gamma} \, \frac{\beta}{\beta-c} \min_{P \in \real_t^0[X]} \int_\alpha^\beta P^2(t) \frac{\mu_\tau(x)}{x} \, \mathrm{d}x \less \mathcal{L}_{f_h,t}^* \less \frac{c \tau}{(1-\tau)\gamma} \, \frac{\alpha}{\alpha-c} \min_{P \in \real_t^0[X]} \int_\alpha^\beta P^2(t) \frac{\mu_\tau(x)}{x} \, \mathrm{d}x\,.
\end{align*}
From Lemma~\ref{LemmaGenericConvergenceRateGaussian}, we know that $\mathcal{L}_{\mu_\tau,t}^* = (1-\tau) \min_{P \in \real_t^0[X]} \int_\alpha^\beta P^2(t) \frac{\mu_\tau(x)}{x} \, \mathrm{d}x$. Thus,
\begin{align*}
    \frac{c \tau}{(1-\tau)^2\gamma}\,\frac{\beta}{\beta-c} \mathcal{L}_{\mu_\tau,t}^* \less \mathcal{L}_{f_h,t}^* \less \frac{c \tau}{(1-\tau)^2\gamma}\,\frac{\alpha}{\alpha-c} \mathcal{L}_{\mu_\tau,t}^*\,.
\end{align*}
From Theorem~\ref{TheoremOptimalAlgorithmGaussian}, we know that $\mathcal{L}_{\mu_\tau,t}^* = \tau^t$. Thus, we obtain that
\begin{align*}
    \mathcal{L}_{f_h,t}^* \asymp \tau^t\,.
\end{align*}
A simple calculation gives that $\tau = \frac{1-\xi}{1-\gamma} \rho$, which yields the claimed result. As for the Gaussian case, an exact calculation of $\mathcal{L}_{f_h,t}^*$ is actually possible. But, after investigation, the resulting expression is lengthy and fairly difficult to simplify, whereas we are primarily interested in the scaling in terms of the iteration number $t$.

\section{Proofs of intermediate results}

\subsection{Proof of Lemma~\ref{LemmaPolynomialRecursion}}
\label{ProofLemmaPolynomialRecursion}

Suppose that $\{x_t\}$ is generated by a first-order method~\eqref{EqnFirstOrder}. Fix $t \gre 1$, then there exists $\alpha_{0,t}, \dots, \alpha_{t-1,t}$ such that
\begin{align}
\label{EqnLemma1Proof1}
    x_t = x_{t-1} + \sum_{j=0}^{t-1} \alpha_{j,t} H_S^{-1} A^\top (A x_j-b)\,.
\end{align}
Multiplying both sides of~\eqref{EqnLemma1Proof1} by $U^\top A$, subtracting $U^\top A \xstar$ and using the normal equation $A^\top A \xstar = A^\top b$, we find that
\begin{align}
\label{EqnLemma1Proof2}
    \Delta_t = \Delta_{t-1} + \sum_{j=0}^{t-1} \alpha_{j,t} C_S^{-1} \Delta_j\,.
\end{align}
First, we aim to show that there exists a polynomial $p_t \in \real_t^0[X]$ such that $\Delta_t \!=\! p_t\!\left(C_S^{-1}\right) \Delta_0$. We proceed by induction over $t \gre 0$. For $t\!=\!0$, the claim is true. Suppose that for some $t \gre 1$, it holds that $\Delta_j = p_j\!\left(C_S^{-1}\right) \Delta_0$ with $p_j \in \real_j^0[X]$ for $j\!=\!0,\dots,t-1$. Then, we have from~\eqref{EqnLemma1Proof2} that
\begin{align}
    \Delta_t &= p_{t-1}\!\left(C_S^{-1}\right)\Delta_0 + \sum_{j=0}^{t-1} \alpha_{j,t} C_S^{-1} p_j\!\left(C_S^{-1}\right) \Delta_0\\
    &= \left(p_{t-1}\!\left(C_S^{-1}\right) +  \sum_{j=0}^{t-1} \alpha_{j,t} C_S^{-1} p_j\!\left(C_S^{-1}\right)\right) \Delta_0 \label{EqnLemma1Proof3}\,.
\end{align}
We set $p_t(x) = p_{t-1}(x) + \sum_{j=0}^{t-1} \alpha_{j,t} x p_j(x)$. It holds that $p_t(0) = p_{t-1}(0) + 0 = 1$, and $\mathrm{deg}(p_t) \less t$ since $\mathrm{deg}(p_{t-1}) \less t-1$ and $\mathrm{deg}(xp_j(x)) \less j+1 \less t$ for $j = 0,\dots,t-1$. Then, from~\eqref{EqnLemma1Proof3}, we have $\Delta_t = p_t\!\left(C_S^{-1}\right) \Delta_0$, which concludes the induction.

Second, we aim to show that $\lim_{n \to \infty} \frac{\Exs[\|\Delta_t\|^2]}{\Exs[\|\Delta_0\|^2]} = \int_\real p^2\left(\lambda^{-1}\right) \mathrm{d}\mu(\lambda)$, where $\mu$ is the l.s.d.~of $C_S$ for $S$ an $m \times n$ Gaussian or SRHT embedding. The Gaussian case is straightforward to prove, by using the rotational invariance of the Gaussian distribution. The SRHT case is more involved, and we leverage tools from free probability theory.

\subsubsection{The Gaussian case}

Let $S$ be an $m \times n$ random matrix with i.i.d.~entries $\mathcal{N}(0,1/m)$. Then, by rotational invariance, $SU$ is an $m \times d$ matrix with i.i.d.~entries $\mathcal{N}(0,1/m)$. Write the eigenvalue decomposition $C_S = V \Sigma V^\top$ where $V$ is a $d \times d$ orthogonal matrix, and $\Sigma$ a diagonal matrix with positive entries $\lambda_1, \dots, \lambda_d$. A standard result states that $V$ and $\Sigma$ are independent matrices, and $V$ is Haar-distributed. 

Fix $t \gre 0$, and let $p_t \in \real_t^0[X]$ such that $\Delta_t = p_t\!\left(C_S^{-1}\right) \Delta_0$. Taking the squared norm and the expectation, we obtain that
\begin{align*}
    \Exs\left[\|\Delta_t\|^2\right] &= \Exs\left[\Delta_0^\top p_t^2\!\left(C_S^{-1}\right) \Delta_0\right]\\
    &= \Exs\left[\Delta_0^\top V p_t^2\!\left(\Sigma^{-1}\right) V^\top \Delta_0\right]\,.
\end{align*}
Using the independence of $\Sigma$, $V$ and $\Delta_0$ and writing $V = [v_1, \dots, v_d]$, we further obtain that
\begin{align*}
    \Exs\left[\|\Delta_t\|^2\right] &= \Exs\!\left[\Delta_0^\top V \Exs\left[p_t^2\!\left(\Sigma^{-1}\right)\right] V^\top \Delta_0\right]\\
    &=  \sum_{i=1}^d \Exs\!\left[(v_i^\top \Delta_0)^2\right] \Exs[p_t^2(\lambda_i^{-1})]\,.
\end{align*}
Since each $v_i$ is uniformly distributed on the unit sphere, we have that $\Exs \left[(v_i^\top \Delta_0)^2\right] = \frac{1}{d}\Exs{\|\Delta_0\|}^2$, so that
\begin{align*}
    \Exs\left[\|\Delta_t\|^2\right] &= \frac{1}{d} \Exs{\|\Delta_0\|}^2 \Exs\left[\sum_{i=1}^d p_t^2(\lambda_i^{-1})\right]\\
    &= \Exs{\|\Delta_0\|}^2 \, \frac{1}{d} \trace \Exs\!\left[p_t^2\!\left(C_S^{-1}\right)\right]\,.
\end{align*}
Dividing both sides of the above equation by $\Exs{\|\Delta_0\|}^2$ and taking the limit $d \to \infty$, we obtain the claimed result,
\begin{align*}
    \frac{\Exs\left[\|\Delta_t\|^2\right]}{\Exs\left[\|\Delta_0\|^2\right]} = \int_\real p_t^2\!\left(\lambda^{-1}\right) \mathrm{d}\mu(\lambda)\,.
\end{align*}

\subsubsection{The SRHT case}

The SRHT does not satisfy rotational invariance as the Gaussian distribution (or Haar matrices), and we need to use a different approach for this proof, based on \emph{asymptotically liberating sequences of unitary matrices}~\cite{anderson2014asymptotically}.

Let $S$ be an $m \times n$ SRHT embedding. We denote by $\mu$ the l.s.d.~of the matrix $C_S$. Following the same first steps as for the Gaussian case, we have that
\begin{align}
    \Exs\left[\|\Delta_t\|^2\right] = \Exs \trace \left[ p_t^2\!\left(C_S^{-1}\right) \Delta_0 \Delta_0^\top \right] = \Exs \trace \left[ p_t^2\!\left(C_S^{-1}\right) \Sigma_0 \right]\,,
\end{align}
where $\Sigma_0 \defn \Exs \Delta_0 \Delta_0^\top$. Writing $p_t^2(x) = \sum_{k=0}^{t} a_k x^{2k}$, it follows that 
\begin{align}
    \Exs\left[\|\Delta_t\|^2\right] = \sum_{k=0}^{t} a_k \, \Exs \trace \left[ C_S^{-2k} \Sigma_0 \right]\,.
\end{align}
Introducing the matrix $\widetilde \Sigma_0 = \frac{\Sigma_0}{\trace \Sigma_0 / d}$, we further obtain
\begin{align}
\label{EqnIntermediateSRHT1}
    \frac{\Exs\left[\|\Delta_t\|^2\right]}{\Exs \left[\|\Delta_0\|^2\right]} = \sum_{k=0}^{t} a_k \,\frac{1}{d} \Exs \trace \left[ C_S^{-2k} \widetilde \Sigma_0 \right]\,.
\end{align}
We use the following result, whose proof leverages some notions from free probability theory. We defer the proof to Appendix~\ref{ProofLemmaAsymptoticDecouplingSRHT}.
\begin{lemmaappendix}
\label{LemmaAsymptoticDecouplingSRHT}
It holds that for any $k \gre 0$, 
\begin{align}
    \lim_{n \to \infty} \frac{1}{d} \Exs \trace \left[ C_S^{-2k} \widetilde \Sigma_0 \right] = \lim_{n \to \infty}\frac{1}{d} \Exs \trace \left[ C_S^{-2k} \right] = \int_\real \lambda^{-2k} \mathrm{d}\mu(\lambda)\,.
\end{align}
\end{lemmaappendix}
Combining~\eqref{EqnIntermediateSRHT1} and the result of Lemma~\ref{LemmaAsymptoticDecouplingSRHT}, it follows that
\begin{align}
    \lim_{n \to \infty} \frac{\Exs\left[\|\Delta_t\|^2\right]}{\Exs \left[\|\Delta_0\|^2\right]} = \sum_{k=0}^t a_k \int_\real \lambda^{-2k} \mathrm{d}\mu(\lambda) = \int_\real p_t^2\!\left(\lambda^{-1}\right) \mathrm{d} \mu(\lambda)\,,
\end{align}
which is the claimed result.

\subsection{Proof of Lemma~\ref{LemmaAsymptoticDecouplingSRHT}}
\label{ProofLemmaAsymptoticDecouplingSRHT}

We introduce a few needed concepts from free probability that will be used in this proof. We refer the reader to~\citep{voiculescu1992free,hiai2006semicircle,nica2006lectures,anderson2010introduction} for an extensive introduction to this field. Consider the algebra $\mathcal{A}_n$ of $n \times n$ random matrices. For $\!X_n\! \in \mathcal{A}_n$, we define the linear functional $\tau_n(X_n) \!\defn\! \frac{1}{n}\Exs\left[\trace X_n\right]$. Then, we say that a family $\{X_{n,1}, \dots, X_{n,I}\}$ of random matrices in $\mathcal{A}_n$ is \emph{asymptotically free} if for every $i \in \{1,\dots,I\}$, $X_{n,i}$ has a limiting spectral distribution, and if $\tau_n\left(\prod_{j=1}^m P_j\left(X_{n,i_j}-\tau\left(P_j(X_{n,i_j})\right)\right)\right) \rightarrow 0$ almost surely for any positive integer $m$, any polynomials $P_1, \dots, P_m$ and any indices $i_1, \hdots, i_m \in \{1,\dots,I\}$ with $i_1 \!\neq\!i_2, \dots,i_{m-1}\!\neq\!i_m$. 

Let $S$ be an $n \times n$ SRHT embedding (we consider the SRHT before discarding its zero rows). By definition, we can write $S=B H W$, where $B$ is an $n \times n$ matrix with i.i.d.~Bernoulli entries on the diagonal, with success probability $m/n$, $H = H_n$ is the $n$-th Walsh-Hadamard matrix. The matrix $W$ is an $n \times n$ bi-signed permutation, i.e., $W = D P$, where $D$ is a diagonal matrix with i.i.d.~random signs, and $P$ is an $n \times n$ uniformly random permutation matrix.

We aim to show that for any $k \gre 0$,
\begin{align}
\label{EqnGoalLemmaB2}
    \lim_{d \to \infty} \tau_d\left( C_S^{-k} \widetilde \Sigma_0 \right) = \lim_{d \to \infty} \tau_d\left( C_S^{-k}\right)\,.
\end{align}
We reduce the problem of proving~\eqref{EqnGoalLemmaB2} to the following, which is more simple to treat. The proof of this reduction is deferred to Appendix~\ref{ProofLemmaReductionDecoupling}.
\begin{lemmaappendix}
\label{LemmaReductionDecoupling}
Suppose that for any $k \gre 0$, we have 
\begin{align}
\label{EqnNoninversemoments}
    \lim_{d \to \infty} \tau_d\left( C_S^k \widetilde \Sigma_0 \right) = \lim_{d \to \infty} \tau_d\left( C_S^k\right)\,.
\end{align}
Then, the claim~\eqref{EqnGoalLemmaB2} is true for any $k \gre 0$.
\end{lemmaappendix}
Thus, we aim to show~\eqref{EqnNoninversemoments} for all $k \gre 0$.

It holds that 
\begin{align*}
C_S = U^\top S^\top S U &= (U^\top W^\top H B)(B H W U)\\
&= U^\top W^\top H B^2 HWU \\
&= U^\top W^\top H B HW U\,,
\end{align*}
where we used $B^2 = B$ in the fourth equality. Further, we have the following equality in distribution, whose proof is deferred to Appendix~\ref{ProofEqualityDistribution}.
\begin{lemmaappendix}
\label{LemmaSomeDistributionEq}
It holds that
\begin{align}
\label{EqnSomeDistributionEq}
    U^\top W^\top H B H W U \overset{\mathrm{d}}{=} U^\top W^\top H W B W^\top H W U\,.
\end{align}
Consequently,
\begin{align}
    C_S \overset{\mathrm{d}}{=} U^\top W^\top H W B W^\top H W U\,.    
\end{align}
\end{lemmaappendix}
Let $k \gre 0$. We have $W^\top W = I_n$, $U^\top U = I_d$, $B^2=B$, $H^2=H$ and $\tau_d(\widetilde \Sigma_0)=1$. Using~\eqref{EqnSomeDistributionEq}, we find
\begin{align}
\label{EqnSomeIneq111}
    \tau_d\!\left( C_S^k \widetilde \Sigma_0 \right) = \tau_d\!\left( (U^\top W^\top H W B W^\top H W U)^k \widetilde \Sigma_0 \right) = \frac{n}{d} \cdot \tau_n\!\left( X_1 (Y X_2)^k \right)\,, 
\end{align}
where we introduced the matrices $X_1 \defn W U \widetilde \Sigma_0 U^\top W^\top$, $X_2 \defn W UU^\top W^\top$ and $Y \defn HWBW^\top H$. These matrices satisfy the following collection of properties, whose proof is deferred to Appendix~\ref{ProofLemmaCollectionProperties}.
\begin{lemmaappendix}
\label{LemmaCollectionProperties}
It holds that $X_1X_2 = X_2 X_1 = X_1$, $X_2^2=X_2$, $Y^2=Y$,
\begin{align}
    \lim_{n \infty} \tau_n(X_1) = \lim_{n \infty} \tau_n(X_2)\,,
\end{align}
and the sets of matrices $\{X_1,X_2\}$ and $\{Y\}$ are asymptotically free. 

Further, for any $k \gre 1$, we have
\begin{align}
    \lim_{n \to \infty} \tau_n( X_1 (YX_2)^k ) = \lim_{n \to \infty} \tau_n(X_2(YX_2)^k)\,.
\end{align}
\end{lemmaappendix}
Now, observe that
\begin{align*}
    \tau_n(X_2(YX_2)^k) &= \tau_n( WUU^\top W^\top(HWBW^\top H W UU^\top W^\top)^k )\\
    &= \frac{d}{n}\tau_d( (U^\top W^\top HWBW^\top H W U)^k )\\
    &= \frac{d}{n} \tau_d(C_S^k)\,,
\end{align*}
where we used the commutativity of the trace in the second equality, and the equality in distribution~\eqref{EqnSomeDistributionEq} for the third equality. Consequently, 
\begin{align}
\label{EqnSomeIneq112}
    \lim_{n \to \infty} \tau_n( X_1 (YX_2)^k ) = \gamma \lim_{d \to \infty} \tau_d(C_S^k)\,.
\end{align}
Combining the above equality~\eqref{EqnSomeIneq112} with equality~\eqref{EqnSomeIneq111}, we obtain the claimed result~\eqref{EqnNoninversemoments}.

\subsection{Proof of Lemma~\ref{LemmaSomeDistributionEq}}
\label{ProofEqualityDistribution}

Note that both $B$ and $D$ are diagonal matrices whose diagonal entries are i.i.d.~random variables, and $P$ is a permutation matrix. Define $\tilde{B}=P B P^\top$ and $\tilde{D}=P^\top DP$, then $\tilde{B}\stackrel{d}{=}B$, $\tilde{D}\stackrel{d}{=}D$, 
\begin{align}
\label{EqnDP=PD}
    DP=P\tilde{D},\quad P^\top D=\tilde{D}P^\top\,.
\end{align}
It follows that
\begin{align*}
U^\top W^\top H W B W^\top H W U &= U^\top P^\top D H D P B P^\top D H D P U\\
&= U^\top P^\top D H P \tilde{D} B \tilde{D} P^\top H D P U\\
&=U^\top  P^\top DH_n PB\tilde{D}^2 P^\top H_n DP U\\
&=U^\top  P^\top DH_n P BP^\top H_n DP U\\
&=U^\top  P^\top DH_n\tilde{B} H_n DP U\\
&\stackrel{d}{=}U^\top P^\top DH_n B H_n DP U,
\end{align*}
where the second equation follows from \eqref{EqnDP=PD}, the third equation holds because $\tilde{D}$ and $B$ are diagonal so they commute, while the fourth equation holds because $\tilde{D}^2 = I_n$.

\subsection{Proof of Lemma~\ref{LemmaOrthogonalFamilyGaussian} -- Alternative construction of the polynomials $\{\Pi_k\}$}
\label{ProofLemmaOrthogonalFamilyGaussian}
We recall that for a polynomial $P$ and a measure (resp.~density) $\mu$, we will denote $\mu[P] \defn \int_\real P(x) \mu(x) \mathrm{d}\mu(x)$ (resp.~$\mu[P] \defn \int_\real P(x)\mu(x) \mathrm{d}x$). Thus, for a density $\mu$, the reader should be aware that $\mu[x]$ and $\mu(x)$ refer to different quantities. 

We present an alternative construction of the orthogonal family $\{\Pi_k\}$ with respect to $\mu_\rho$, explicitly based on the Chebyshev polynomials of the second kind. This explicit construction allows us to leverage several properties of the polynomials $\{\Pi_k\}$ which are useful to perform calculations and prove Lemma~\ref{LemmaOrthogonalFamilyGaussian}, as well as Theorem~\ref{TheoremOptimalAlgorithmGaussian}. 

We introduce the shifted Chebyshev polynomials of the second kind, which are defined by the recurrence
\begin{align}
\label{EqnRecursionPolynomialsShiftedChebyshev}
    Q_0(x) = 1\,,\quad Q_1(x) = \frac{x-(1\!+\!\rho)}{\sqrt{\rho}}\,,\quad Q_{k+1}(x) = \frac{x-(1\!+\!\rho)}{\sqrt{\rho}}\, Q_k(x) - Q_{k-1}(x)\,.
\end{align}
A standard result states that the polynomials $Q_k$ are orthonormal with respect to the measure $\nu(x) \mathrm{d}x \defn x \mu_\rho(x) \mathrm{d}x$. We set $\widehat \Pi_0(x)=1$, and for $k \gre 1$,
\begin{align}
\label{EqnOrthogonalPolynomialsMP}
    \widehat \Pi_k(x) \defn 1 - \sum_{j=1}^k (-1)^{j-1} \sqrt{\rho}^{j-1} \, x Q_{k-1}(x)\,.
\end{align}
For instance, we have $\widehat \Pi_1(x) = 1- x$ and $\widehat \Pi_2(x) = 1 - (2+\rho) x + x^2$. 

We aim to show that $\{\widehat \Pi_k\}$ is an orthogonal family with respect to $\mu_\rho$ and then, that $\widehat \Pi_k = \Pi_k$.

First, we show that the polynomials $\widehat \Pi_k$ form an orthonormal family with respect to $\mu_\rho$ such that $\mathrm{deg}(\widehat \Pi_k)\!=\!k$ and $\widehat \Pi_k(0)\!=\!1$. For $k\!\gre\!1$, we define the polynomial $T_k(x)\!=\!x Q_{k-1}(x)$ and the measure $\lambda_\rho(x)\!=\!x^{-1}\mu_\rho(x)$. We have that $\lambda_\rho[T_kT_\ell]\!=\!\nu_\rho[Q_{k-1}Q_{\ell-1}]\!=\!\delta_{k\ell}$, so that the $T_k$ are orthonormal with respect to $\lambda_\rho$. Since $\mathrm{deg}(Q_{k-1})\!=\!k\!-\!1$, we have $\mathrm{deg}(T_k)\!=\!k$. We also have $T_k(0)=0 \cdot Q_{k-1}(0)\!=\!0$.

Second, we show that $\mu_\rho[Q_k]\!=\!(-1)^k \sqrt{\rho}^k$, which will immediately imply that 
\begin{align}
\lambda_\rho[T_k] = \lambda_\rho[xQ_{k-1}(x)] = \mu_\rho[Q_{k-1}] = (-1)^{k-1}\sqrt{\rho}^{k-1}\,.
\end{align}
We denote $u_k \!\defn\! \mu_\rho[Q_k]$. The measure $\mu_\rho$ is a probability measure, so that $u_0=1$. Further, we have 
\begin{align*}
    u_1 = \mu_\rho[Q_1] = \int_a^b \frac{x-(1+\rho)}{\sqrt{\rho}} \mu_\rho(x)\, \mathrm{d}x = \frac{-1-\rho + \int_a^b x \mu_\rho(x)\,\mathrm{d}x }{\sqrt{\rho}}\,.
\end{align*}
The first moment $\mu_\rho[x]$ is equal to $1$, so that $u_1=-\sqrt{\rho}$. From the recurrence relationship~\eqref{EqnRecursionPolynomialsShiftedChebyshev}, we obtain $u_{k+1}=-\frac{1+\rho}{\sqrt{\rho}}u_k - u_{k-1}$. The characteristic equation $x^2+\frac{1+\rho}{\sqrt{\rho}}x+1=0$ has roots $-1/\sqrt{\rho}$ and $-\sqrt{\rho}$. Therefore, $u_k=\alpha \frac{(-1)^k}{\sqrt{\rho}^k}+\beta (-1)^k \sqrt{\rho}^k$ for some $\alpha, \beta \in \real$. Using the initial values $u_0$ and $u_1$, we find $\alpha=0$ and $\beta=1$. This yields the claimed formula for $u_k$. 

Then, using the definition~\eqref{EqnOrthogonalPolynomialsMP}, we have
\begin{align}
    \widehat \Pi_k &= 1 - \sum_{j=1}^k (-1)^{j-1} \sqrt{\rho}^{j-1} \, x Q_{k-1}(x)\\
    &= 1 - \sum_{j=1}^k \lambda_\rho[T_k] \, T_k(x)\,.
\end{align}
Hence, recognizing the Gram-Schmidt orthogonalization of the constant polynomial $1$ with respect to $\{T_1, \dots, T_k\}$, we deduce that the family $\{\widehat \Pi_k,T_1,\dots,T_k\}$ is orthogonal with respect to $\lambda_\rho$, and is a basis of $\real_k[X]$. Consider now the variational problem
\begin{align}
\label{EqnVariationalProblem2}
    \min_{p \in \real_k^0[X]} \int p^2(x) \lambda_\rho(x) \, \mathrm{d}x\,.
\end{align}
Let $p \in \real_k^0[X]$ and decompose $p$ as $p\!=\!\alpha_0 \widehat \Pi_k \!+\! \sum_{j=1}^k \alpha_j T_k$. Using $p(0)\!=\!1$, $\widehat \Pi_k(0)\!=\!1$ and $T_j(0)\!=\!0$, we get that $\alpha_0$ must be equal to $1$. Then,
\begin{align*}
    \int p^2(x) \lambda_\rho(x) \, \mathrm{d}x &= \int \widehat \Pi_k^2(x) \lambda_\rho(x) \, \mathrm{d}x + 2 \sum_{j=1}^k \alpha_j \int \widehat \Pi_k(x) T_j(x) \lambda_\rho(x) \,\mathrm{d}x\\
    & \quad + \int (\sum_{j=1}^k \alpha_j T_j(x))^2 \lambda_\rho(x) \,\mathrm{d}x\,.
\end{align*}
The cross-term is equal to $0$ by orthogonality of the family $\{\widehat \Pi_k, T_1, \dots, T_k\}$. The third-term is non-negative, and equal to $0$ if and only if $p = \widehat \Pi_k$. Therefore, the minimizer of the variational problem~\eqref{EqnVariationalProblem2} is exactly $\widehat \Pi_k$. On the other hand, applying Lemma~\ref{LemmaOptimalSolutionPolynomialOptimizationProblem} with $\nu\!=\!x\lambda_\rho\!=\!\mu_\rho$, we know that the solution of each of the problems~\eqref{EqnVariationalProblem2} (for varying $k$) is unique, and the solutions form an orthogonal family with respect to $x \lambda_\rho(x)\mathrm{d}x\!=\!\mu_\rho(x)\mathrm{d}x$. Thus, we obtain that the family $\{\widehat \Pi_k\}$ is orthogonal with respect to $\mu_\rho$.

Finally, we show that the sequence $\{\widehat \Pi_k\}$ satisfies the recurrence relationship~\eqref{EqnRecurrenceOrthogonalPolynomialsGaussian}. Observe that
\begin{align*}
    x \widehat \Pi_k(x) = x - \sum_{j=1}^k \lambda_\rho[T_j] x T_j(x)
    &= x - \lambda_\rho[T_1] x T_1(x) - \sum_{j=2}^k \lambda_\rho[T_j] x T_j(x)\\
    &= x - x^2 - \sum_{j=2}^k \lambda_\rho[T_j] x T_j(x)\,.
\end{align*}
Multiplying~\eqref{EqnRecursionPolynomialsShiftedChebyshev} by $x$ and using the definition $T_k(x)=x Q_{k-1}(x)$, we find that for $k \gre 2$,
\begin{align*}
    x \, T_j(x) = \sqrt{\rho} \left(T_{j-1}(x) + T_{j+1}(x)\right) + (1+\rho) T_j(x)\,.
\end{align*}
Using the above decomposition of $xT_j(x)$, it obtain $\sum_{j=2}^n \lambda_\rho(T_j) x T_j(x) = s_1 + s_2 + s_3$, where 
\begin{align*}
    s_1 \defn \sqrt{\rho}\, \sum_{j=2}^k \lambda_\rho(T_j) T_{j+1}(x) &= \sum_{j=2}^k (-1)^{j-1} \sqrt{\rho}^j T_{j+1}(x)\\ 
    &= \sum_{j=3}^{k+1} (-1)^{j} \sqrt{\rho}^{j-1} T_j(x)\\
    &= \widehat \Pi_{k+1}(x) - 1 + T_1(x) - \sqrt{\rho} \, T_2(x)\\ &= \widehat \Pi_{k+1}(x) - 1 + x - x^2 + (1+\rho)\,x\,,
\end{align*}
the second term is $s_2 \defn \sqrt{\rho}\! \sum_{j=2}^k \lambda_\rho[T_j] T_{j-1}(x) \!=\! \sum_{j=2}^k (-1)^{j-1} \sqrt{\rho}^j T_{j-1}(x) \!=\! \rho \! \left(\widehat \Pi_{k-1}(x)\!-\!1\right)$ and the third term is $s_3 \defn (1+\rho) \! \sum_{j=2}^k \lambda_\rho[T_j] T_j(x) \!=\!-(1\!+\!\rho) \! \left(\widehat \Pi_k(x) \!-\! 1\! +\! x\right)$. Consequently,
\begin{align*}
    x \Pi_k(x) &= x - x^2 - s_1 - s_2 - s_3\\
    &= x - x^2 - \widehat \Pi_{k+1}(x) + 1 - x + x^2 - (1+\rho) x - \rho\,(\widehat \Pi_{k-1}(x) \!-\! 1) + (1+\rho) (\widehat \Pi_k(x) \!-\! 1 \!+\! x)\\
    &= - \widehat \Pi_{k+1}(x) - \rho \, \widehat \Pi_{k-1}(x) + (1+\rho) \widehat \Pi_k(x)\,,
\end{align*}
which is the claimed recurrence. We deduce that $\widehat \Pi_k = \Pi_k$, and that the family $\{\Pi_k\}$ is orthogonal with respect to $\mu_\rho$.

\section{Description of numerical experiments}
\label{SectionExperimentalSetup}

Numerical simulations are run in \textit{Python} with the numerical linear algebra module \textit{NumPy} and the scientific computation module \textit{SciPy}, on a machine with $256$Gb of memory.

To generate an $m \times n$ Haar matrix $S_h$, we sample an $m \times n$ matrix $G$ with i.i.d.~Gaussian entries $\mathcal{N}(0,1)$, and we set $S_h$ to be its $m \times n$ matrix of right singular vectors. To generate an $m \times n$ SRHT matrix, we follow the description given in Section~\ref{SectionIntroduction}. The plots correspond to one trial for each embedding.

\subsection{Figure~\ref{FigSRHTDensity}}

We set $n=8192$, $d=1640$ and $m \in \{1720, 3280, 4915\}$. We generate the plots of $\mu_\rho$ and $f_{h,r}$ by discretizing their respective supports with step size $1e\!-\!5$.

\subsection{Figures~\ref{FigComparisonIHS} and~\ref{FigComparisonRates}}

We generate an $n \times d$ Gaussian matrix $G$ with i.i.d.~entries, and we compute its left singular matrix $U$ and right singular matrix $V$. Then, we set $A = U \Sigma V^\top$, where $\Sigma$ is a $d \times d$ diagonal matrix with entries $\Sigma_j = 0.98^j$ for $j = 1,\dots, d$. We generate a vector $b$ using a planted model $b = A x_\text{pl} + \frac{1}{\sqrt{n}} \mathcal{N}(0,I_n)$, and $x_\text{pl} \sim \frac{1}{\sqrt{d}}\mathcal{N}(0, I_d)$. Note that, although the performance of the algorithms do not depend on the data $A$ and $b$, we choose a standard statistical model to generate the data, and a data matrix with a very large condition number.

Algorithms~\ref{AlgorithmOptimalFirstOrderGaussian} and~\ref{AlgorithmOptimalFirstOrderHaar} are implemented following their pseudo-code description. We use small perturbations of the algorithmic parameters by setting $a^\delta_t = (1+\delta)a_t$ and $b_t^\delta = (1-\delta) b_t$ with $\delta=0.01$ -- where $a_t$ and $b_t$ correspond to the parameters as described in Theorem~\ref{TheoremGeneralOptimalFirstOrderMethod}. Similarly, for the Heavy-ball method with fixed SRHT embeddings and parameters derived based on our new asymptotitc edge eigenvalues ("SRHT (edge eig.)"), we use instead the slightly perturbed edge eigenvalues $\lambda^\delta_h = (1-\delta) \lambda_h$ and $\Lambda^\delta_h = (1+\delta) \Lambda_h$, with $\delta=0.01$. These small perturbations of the parameters are necessary in practice due to the finite-sample approximations. For the Heavy-ball method with fixed SRHT embeddings based on the bounds of~\citet{tropp2011improved} ("SRHT (baseline)"), we use the parameters prescribed in~\cite{lacotte2019faster}. For the Heavy-ball method with refreshed SRHT embeddings ("SRHT (refreshed)"), we use the parameters prescribed in~\cite{lacotteiterative20}. For each algorithm, results are averaged over $20$ independent trials (using the same data $A$ and $b$).

\section{Proofs of auxiliary results}

\subsection{Proof of Lemma~\ref{LemmaReductionDecoupling}}
\label{ProofLemmaReductionDecoupling}

Suppose that~\eqref{EqnNoninversemoments} holds. Let $k \gre 0$. We have 
\begin{align}
    C_S^{-k} = (I_d - (I_d - C_S))^{-k} = \left(\sum_{j=0}^{\infty} (I_d - C_S)^j \right)^k\,, 
\end{align}
where the series expansion $(I_d - (I_d - C_S))^{-1} = \sum_{j=0}^{\infty} (I_d - C_S)^j$ holds almost surely, due to the fact that $C_S$ has spectrum within $(0,1)$ almost surely. There exist coefficients $\{a_\ell\}$ such that $\left(\sum_{j=0}^{\infty} x^j \right)^k = \sum_{\ell=0} a_\ell x^\ell$, and such that the sum is absolutely convergent, i.e., $\sum_{\ell=0} |a_\ell| |x|^\ell < +\infty$, for any $x\in (0,1)$. Consequently,
\begin{align}
    C_S^{-k} = \sum_{\ell=0}^\infty a_\ell C_S^\ell
\end{align}
Then, by absolute convergence of $\sum_\ell a_\ell x^\ell$ and using the fact that $\|C_S\|_2 < 1$, we can exchange the operator $\tau_d$ and the infinite sum, so that
\begin{align}
    \tau_d\!\left(C_S^{-k}\right) = \tau_d\!\left(\sum_{\ell=0}^\infty a_\ell C_S^\ell\right) = \sum_{\ell=0}^{\infty} a_\ell \tau_d\!\left(C_S^\ell\right)\,.
\end{align}
and writing the latter as a series in $C_S$, we obtain the claimed result. Due to the fact that $\sup_\ell \lim_{d \infty} \tau_d\!\left(C_S^\ell\right) < 1$, and using again the absolute convergence of $\sum_\ell a_\ell x^\ell$ for $|x| < 1$, it follows that
\begin{align}
    \lim_{d \to \infty} \tau_d\!\left(C_S^{-k}\right) &= \lim_{d \to \infty} \sum_{\ell=0}^{\infty} a_\ell \tau_d\!\left(C_S^\ell\right)\\
    &= \sum_{\ell=0}^{\infty} a_\ell \lim_{d \to \infty} \tau_d\!\left(C_S^\ell\right)\\
    &= \sum_{\ell=0}^{\infty} a_\ell \lim_{d \to \infty} \tau_d\!\left(C_S^\ell \widetilde \Sigma_0\right)\,.
\end{align}
Using the same arguments, we find that 
\begin{align}
    \sum_{\ell=0}^{\infty} a_\ell \lim_{d \to \infty} \tau_d\!\left(C_S^\ell \widetilde \Sigma_0\right) = \lim_{d \to \infty} \tau_d\!\left(C_S^{-k} \widetilde \Sigma_0 \right)\,,
\end{align}
and we conclude that 
\begin{align}
    \tau_d\!\left(C_S^{-k} \widetilde \Sigma_0 \right) = \tau_d\!\left(C_S^{-k}\right)
\end{align}

\subsection{Proof of Lemma~\ref{LemmaCollectionProperties}}
\label{ProofLemmaCollectionProperties}

We have 
\begin{align*}
    X_1 X_2 = W U \widetilde \Sigma_0 U^\top W^\top W U U^\top W^\top = W U \widetilde \Sigma_0 U^\top W^\top = X_1
\end{align*}
where we used in the second equality $U^\top W^\top W U = I_d$. Similarly, we obtain $X_2 X_1 = X_1$.

We have
\begin{align*}
    Y^2 = (HWBW^\top H)(HWBW^\top H) = HWBW^\top H = Y
\end{align*}
where we used in the second equality $BW^\top H H WB = B$.

We have
\begin{align*}
    X_2^2 = W UU^\top W^\top W UU^\top W^\top = W UU^\top W^\top = X_2\,,
\end{align*}
where we used in the second equality $U^\top W^\top W U = I_d$.

Further, it holds that
\begin{align*}
    \lim_{n \infty} \tau_n(X_1) = \gamma \lim_{d \infty} \tau_d(\widetilde \Sigma_0) = \gamma = \lim_{n \infty} \tau_n(X_2),.
\end{align*}
We show asymptotic freeness. Note that the matrices $UU^\top$, $B$ and $\widetilde \Sigma_0$ have l.s.d.~compactly supported. For the latter, this directly follows from our initial assumption that the condition number of the matrix $U^\top A\Exs[x_0 x_0^\top] A^\top U + U^\top bb^\top U$ remains bounded. Then, applying Corollary~3.2 from~\cite{anderson2014asymptotically} with the set of asymptotically liberating matrices $\{W, HW\}$, we immediately obtain asymptotic freeness of $\{X_1, X_2\}$ and $\{Y\}$.

It remains to show that for any $k \gre 0$, 
\begin{align}
\label{EqnPropertyNCCalculus}
    \lim_{n \to \infty} \tau_n( X_1 (YX_2)^k ) = \lim_{n \to \infty} \tau_n(X_2 (YX_2)^k)\,.
\end{align}
For the rest of this proof, we use the more compact notations $a \defn X_1$, $b\defn Y$, $c \defn X_2$ and $\varphi = \lim_{n \to \infty} \tau_n$. We show~\eqref{EqnPropertyNCCalculus} by induction over $k \gre 0$. For $k=0$, the claim is true because $\varphi(a)=\varphi(c)$ as shown above. Fix $k \gre 1$ and suppose that the claim is true for $j=0,\dots,k-1$. By asymptotic freeness, we have
\begin{align}
    \varphi\Big((a-\varphi(a)) \big((b-\varphi(b))(c-\varphi(c))\big)^k\Big) = 0\,.
\end{align}
We expand the left-hand side of the above equation as
\begin{align*}
    & \quad \varphi\Big((a-\varphi(a)) \big((b-\varphi(b))(c-\varphi(c))\big)^k\Big)\\ = & \quad \varphi\big(a(bc)^k\big) + \sum_{\substack{\delta_1,\dots,\delta_{2k} \in \{0,1\}\\(\delta_1,\dots,\delta_{2k}) \neq (1,\dots,1)}} \varphi\Big( a b^{\delta_1} c^{\delta_2} b^{\delta_3} \dots c^{\delta_{2k}} (-\varphi(b))^{1-\delta_{1}} \dots (-\varphi(c))^{1-\delta_{2k}} \Big)\\
    &= \quad \varphi\big(a(bc)^k\big) + \sum_{\substack{\delta_1,\dots,\delta_{2k} \in \{0,1\}\\(\delta_1,\dots,\delta_{2k}) \neq (1,\dots,1)}} (-\varphi(b))^{1-\delta_{1}} \dots (-\varphi(c))^{1-\delta_{2k}} \varphi\Big( a b^{\delta_1} c^{\delta_2} \dots c^{\delta_{2k}} \Big)\,.
\end{align*}
For binary exponents $(\delta_1, \dots, \delta_{2k}) \neq (1, \dots,1)$, the product of non-commutative matrices $b^{\delta_1} c^{\delta_2} b^{\delta_3} \dots c^{\delta_{2k}}$ must have a sub-product of the form $bb$ or $cc$. Using the fact that $b^2 = b$ and $c^2=c$, it follows that there exists some integer $\ell$ such that $0 \less \ell < k$, and 
\begin{align*}
    b^{\delta_1} c^{\delta_2} b^{\delta_3} \dots c^{\delta_{2k}} = (bc)^\ell\,.
\end{align*}
Using the induction hypothesis, we have
\begin{align*}
    \varphi\Big( a b^{\delta_1} c^{\delta_2} \dots c^{\delta_{2k}} \Big) = \varphi(a (bc)^\ell) = \varphi( c (bc)^\ell ) = \varphi\Big(c b^{\delta_1} c^{\delta_2} \dots c^{\delta_{2k}} \Big) \,.
\end{align*}
Consequently, we get
\begin{align*}
    & \quad \varphi\Big((a-\varphi(a)) \big((b-\varphi(b))(c-\varphi(c))\big)^k\Big)\\ = & \quad \varphi\big(a(bc)^k\big) + \sum_{\substack{\delta_1,\dots,\delta_{2k} \in \{0,1\}\\(\delta_1,\dots,\delta_{2k}) \neq (1,\dots,1)}} \varphi\Big( c b^{\delta_1} c^{\delta_2} b^{\delta_3} \dots c^{\delta_{2k}} (-\varphi(b))^{1-\delta_{1}} \dots (-\varphi(c))^{1-\delta_{2k}} \Big)  
\end{align*}
On the other hand, using asymptotic freeness again, we have
\begin{align*}
    0= & \quad \varphi\Big((c-\varphi(c)) \big((b-\varphi(b))(c-\varphi(c))\big)^k\Big)\\ = & \quad \varphi\big(c(bc)^k\big) + \sum_{\substack{\delta_1,\dots,\delta_{2k} \in \{0,1\}\\(\delta_1,\dots,\delta_{2k}) \neq (1,\dots,1)}} \varphi\Big( c b^{\delta_1} c^{\delta_2} b^{\delta_3} \dots c^{\delta_{2k}} (-\varphi(b))^{1-\delta_{1}} \dots (-\varphi(c))^{1-\delta_{2k}} \Big)     
\end{align*}
Combining the two above sets of equalities, we obtain
\begin{align*}
    \varphi\big(a(bc)^k\big) = \varphi\big(c(bc)^k\big)\,,
\end{align*}
which concludes the induction, and the proof.

\end{document}